\documentclass[11pt,twoside]{article}
\usepackage{amsfonts}
\usepackage{amsmath}
\usepackage{amssymb}
\usepackage{multicol}
\usepackage{graphics}
\usepackage{stmaryrd}
\usepackage{cite}
\newtheorem{theorem}{Theorem}[section]
\newtheorem{lemma}[theorem]{Lemma}
\newtheorem{corollary}[theorem]{Corollary}

\newtheorem{definition}[theorem]{Definition}
\newtheorem{remark}[theorem]{Remark}

\numberwithin{equation}{section}
\newenvironment{proof}[1][Proof]{\noindent\textbf{#1.} }{\hfill $\Box$}
\allowdisplaybreaks

 \makeatletter
\begin{document}

\author{Guo Lin\thanks{E-mail: ling@lzu.edu.cn. }  \\
{School of Mathematics and Statistics, Lanzhou University,}\\
{Lanzhou, Gansu 730000, People's Republic of China}}
\title{\textbf{Traveling
Wave Solutions for Integro-Difference Systems}} \maketitle

\begin{abstract}
This paper is concerned with the traveling
wave solutions for  integro-difference systems of higher order. By using Schauder fixed point theorem, the existence of  traveling wave solutions is reduced to the existence of generalized upper and lower solutions. Then the asymptotic behavior of traveling wave solutions is
studied by the idea of contracting rectangles. To illustrate our results, the  traveling wave solutions of three systems are considered, which completes some known results.

\textbf{Keywords}: generalized upper and lower solutions; contracting rectangle; asymptotic behavior; asymptotic spreading

\textbf{AMS Subject Classification (2000)}:  45C05; 45M05; 92D40.
\end{abstract}

\section{Introduction}
\noindent

In this paper, we investigate the existence and asymptotic behavior of traveling wave solutions of the following integro-difference system
\begin{equation}\label{1}
u_{n+1}^i(x)=\int_{\mathbb{R}}P_i[u_{n-\tau+1}^1(y), \cdots, u_{n}^1(y), u_{n-\tau+1}^2(y),\cdots, u_{n}^m(y)]k_i (x-y)dy,
\end{equation}
in which $m\in \mathbb{N}$ and $ \tau\in \mathbb{N}$ are constants, $i\in I=: \{1,2, \cdots, m\}, n\in \mathbb{N} \bigcup \{0\}, x\in \mathbb{R}, u_n^i\in \mathbb{R}, P_i: \mathbb{R}^{m\times \tau}\to \mathbb{R},$ $k_i:\mathbb{R}\to\mathbb{R}^+$ is a probability function or kernel function. Moreover, $P_i$ satisfies the following assumptions:
\begin{description}
\item[(P1)] there exists $\mathbf{M}=(M_1, M_2,\cdots , M_m)$ such that $[\mathbf{0}, \mathbf{M}]$ with $\mathbf{0}=(0,0,\cdots, 0)$ is an invariant region of the corresponding difference system of \eqref{1}, i.e.,
    \[
    0\le P_i [h_{11}, h_{12}, \cdots, h_{m\tau}]\le M_i
    \]
    with
    \[
    0\le h_{ij} \le M_i, i\in I, j\in J=:  \{1, 2, \cdots, \tau\};
    \]
    \item[(P2)] there exists $L>0$ such that
    \[
    |P_i[h_{11},h_{12},\cdots, h_{m\tau}]-P_i[f_{11},f_{12},\cdots, f_{m\tau}]|\le L \sum_{l\in I, j\in J}|h_{lj}-f_{lj}|
    \]
    for any $h_{lj},f_{lj}\in [0,M_l], i, l\in I, j\in J;$
    \item[(P3)] $P_i[0,0,\cdots , 0]=0$ and there exists $\mathbf{E}=(E_1, E_2, \cdots, E_m)$ such that
    \[
P_{i}\left[ \overset{\tau}{\overbrace{E_{1},\cdots ,E_{1}}},\overset{\tau}{%
\overbrace{E_{2},\cdots ,E_{2}}},E_3,\cdots ,E_{m}\right] =E_{i},i\in I;
\]
\item[(P4)] $\mathbf{0}\ll \mathbf{E}\le \mathbf{M}.$
\end{description}
Moreover, for every $i\in I,$ the probability function $k_i$ satisfies the following conditions:
\begin{description}
\item[(k1)] $k_i: \mathbb{R}\to\mathbb{R}^+$ is Lebesgue measurable and integrable;
    \item[(k2)] $k_i: \mathbb{R}\to\mathbb{R}^+$ satisfies $k_i(x)=k_i(-x),x\in\mathbb{R};$
    \item[(k3)] $\int_{\mathbb{R}} k_i (y)dy=1$ and $\int_{\mathbb{R}} k_i (y)e^{\lambda y} dy <\infty$ for any $\lambda \ge 0.$
\end{description}

If $m=1$ and $\tau=1,$ then \eqref{1} becomes
\begin{equation}\label{2}
v_{n+1}(x)=\int_{\mathbb{R}} b(v_n(y))k(x-y)dy,
\end{equation}
in which $b:\mathbb{R}^+\to \mathbb{R}^+$ is bounded and continuous and $b(0)=0$, $k:\mathbb{R}\to \mathbb{R}^+$
satisfies (k1)-(k3). In the past three decades, the traveling wave solutions of \eqref{2} have been widely studied, we refer to Creegan and Lui \cite{creeganlui}, Hsu and Zhao \cite{hsuzhao}, Kot \cite{kot}, Kot et al. \cite{kotlewis}, Liang and Zhao \cite{liang}, Lui \cite{lui,lui3,lui1,lui2}, Neubert and Caswell \cite{neu},
Wang et al. \cite{wangkot}, Weinberger \cite{wein,wein2}, Weinberger et al. \cite{weinberger},  Weinberger et al. \cite{weink} and Yi et al. \cite{yi}. In these papers, the (local) monotonicity of $b$ plays a very important role.

If $\tau=1,$ then Liang and Zhao \cite{liang},  Weinberger et al. \cite{weinberger} and Yi et al. \cite{yi} investigated the propagation modes of \eqref{1} by traveling wave solutions and asymptotic spreading. Similar to the study of scalar equations, the monotonicity of semiflows (see \cite{smith}) is the most essential assumption in \cite{liang,weinberger}. Recently, Lin and Li \cite{linlidc} and Lin et al. \cite{llrjmb} considered the existence of traveling wave solutions of a competitive system ($\tau=1, m=2$) by a cross iteration scheme.

If $\tau=2,m=1,$ and the birth function is (locally) monotone, then the traveling wave solutions and asymptotic spreading were studied by Lin and Li \cite{linlijmaa} and Pan and Lin \cite{panlin}.
When $\tau>1,$ it is possible that standard upper and lower solutions make no sense in the study of traveling wave solutions of \eqref{1}. For example, consider the following model
\begin{equation}\label{3}
v_{n+1}(x)=\int_{\mathbb{R}} \frac{(1+d) v_n(y)}{1+d (v_n(y)+av_{n-1}(y))} k(x-y)dy,
\end{equation}
in which $d >0, a\ge 0$ are constants. If $a=0,$ then the traveling wave solutions of \eqref{3} have been investigated by Kot \cite{kot}. For $a>0,$ it is clear that
\[
\int_{\mathbb{R}} \frac{(1+d) v_n(y)}{1+d (v_n(y)+av_{n-1}(y))} k(x-y)dy
 \]
is monotone increasing in $v_n(\cdot)$ while decreasing in $v_{n-1}(\cdot)$ such that the standard upper and lower solutions cannot be applied. Moreover, \eqref{3} does not satisfy the local monotonicity in Pan and Lin \cite{panlin}. In this paper, we shall introduce the generalized upper and lower solutions to overcome the difficulty that arises from the deficiency of classical comparison principle. Combining the generalized upper and lower solutions with Schauder fixed point theorem, we reduce the existence of positive traveling wave solutions of \eqref{1} to the existence of generalized upper and lower solutions.

The asymptotic behavior of traveling wave solutions is very important since it reflects the transition between different steady states. If a system can generate monotone semiflows, then the behavior can be proved by the monotonicity of traveling wave solutions \cite{liang,weinberger,wuzou2}. Otherwise, the study will be very hard. In
\cite{linlidc,llrjmb}, we obtained the asymptotic behavior by constructing very precise upper and lower solutions of a competitive system with special kernel functions. When the kernel functions are different from those in \cite{linlidc,llrjmb}, then the study is still very complex. Furthermore, if a system, e.g., \eqref{3}, does not admit classical comparison principle, then it is difficult to obtain the asymptotic behavior of traveling wave solutions by constructing auxiliary systems and functions (e.g., upper and lower solutions). Recently, Lin and Ruan \cite{linruan} established the asymptotic behavior of traveling wave solutions of some non-monotonic delayed reaction-diffusion systems via  contracting rectangles. Motivated by the idea in \cite{linruan}, we shall consider the asymptotic behavior of traveling wave solutions of \eqref{1} by the contracting rectangles of the corresponding difference system.

To illustrate our main results, we also consider the traveling wave solutions of equation \eqref{3} and competitive systems of higher order terms. Although these models are not monotone, we still obtain the existence, nonexistence and asymptotic behavior of traveling wave solutions. In particular, we give the minimal wave speed of nontrivial traveling wave solutions formulating the simultaneous invasion of all competitors in population dynamics \cite{shi}. Note that the minimal wave speed of the competitive system just depends on the linearized model near the unstable steady state, which implies that the nonmonotone nonlinear terms may be harmless to the invasion of multi competitive species.

The remainder of this paper is organized as follows. In Section 2, we give some preliminaries. In Section 3, the existence of positive traveling wave solutions is studied by generalized upper and lower solutions. In Section 4, the asymptotic behavior of traveling wave solutions is investigated by contracting rectangles. In the last section, we consider some examples to illustrate our results, which completes some known results.

\section{Preliminaries}
\noindent

In this paper, we shall use the standard partial ordering in $\mathbb{R}^m.$ Define $C$ by
\[
C=\{U(x)| U(x): \mathbb{R}\rightarrow \mathbb{R}^m \text{ is
uniformly continuous and bounded for } x\in\mathbb{R}\}
\]
equipped with the compact open topology.
Moreover, if $\textbf{A}\le \textbf{B}\in \mathbb{R}^m$, then
\[
C_{[\mathbf{A},\mathbf{B}]}=\{ U: U \in C \text{ and  }
\mathbf{A}\leq U(x) \leq \mathbf{B} \text{ for all }x\in
\mathbb{R}\}.
\]
Let $\|\cdot \|$ be the supremum norm in $\mathbb{R}^m$ and $\mu >0$ be a constant. Define
\[
B_{\mu }\left( \mathbb{R},\mathbb{R}^{m}\right) =\left\{\Phi \in
C:\sup_{x \in \mathbb{R}}\| \Phi (x)\|
e^{-\mu \left\vert x \right\vert }<\infty \right\}
\]%
and
\[
\left\vert \Phi \right\vert _{\mu }=\sup_{x \in
\mathbb{R}}\| \Phi (x )\| e^{-\mu \left\vert
x \right\vert }.
\]%
Then  $\left( B_{\mu }\left( \mathbb{R},\mathbb{R}%
^{m}\right) ,\left\vert \cdot \right\vert _{\mu}\right) $ is a
Banach space.

The definition of traveling wave solutions is given as follows.
\begin{definition}
\label{de1}{\rm A traveling wave solution of \eqref{1} is a special
solution with the form
\[
u_n^i(x)=\phi_i (\xi), \xi= x+cn \in\mathbb{R},i\in I,
\]
where $\Phi=(\phi_1,\phi_2, \cdots, \phi_m) \in C$ is the wave profile
that propagates through the one-dimensional spatial domain $\mathbb{R}$ at the
constant wave speed $c>0$.}
\end{definition}
By Definition \ref{de1}, $\phi_i$ and $c$ must satisfy
\begin{eqnarray*}
\phi _{i}(\xi +c)&=&\int_{\mathbb{R}}P_{i}[\phi _{1}(s-\tau c+c),\cdots
,\phi _{1}(s),\phi _{2}(s-\tau c+c),\cdots ,\phi _{m}(s)]k_{i}(\xi -s)ds \\
&=&\int_{\mathbb{R}}P_{i}[\phi _{1}(\xi -s+\tau c-c),\cdots ,\phi _{1}(\xi
-s),\cdots ,\phi _{m}(\xi -s)]k_{i}(s)ds
\end{eqnarray*}
or
\begin{eqnarray}
\phi _{i}(\xi ) &=&\int_{\mathbb{R}}P_{i}[\phi _{1}(s-\tau c+c),\cdots ,\phi
_{1}(s),\phi _{2}(s-\tau c+c),\cdots ,\phi _{m}(s)]k_{i}(\xi -s-c)ds
\nonumber \\
&=&\int_{\mathbb{R}}P_{i}[\phi _{1}(\xi -s-\tau c),\cdots ,\phi _{1}(\xi
-s-c),\cdots ,\phi _{m}(\xi -s-c)]k_{i}(s)ds  \label{4}
\end{eqnarray}
for $i\in I, \xi\in\mathbb{R}.$ Of course, these expressions are equivalent to each other, but \eqref{4} is good at defining an operator, which will be utilized in the next section.

Moreover, we also require that
\begin{equation}\label{5}
\lim_{\xi\to -\infty}\phi_i(\xi)=0, \liminf_{\xi\to + \infty}\phi_i(\xi) >0, i\in I
\end{equation}
or
\begin{equation}\label{500}
\lim_{\xi\to -\infty}\phi_i(\xi)=0, \lim_{\xi\to + \infty}\phi_i(\xi) = E_i, i\in I.
\end{equation}
Then a traveling wave solution satisfying \eqref{4} with \eqref{5} or \eqref{4} with \eqref{500} could formulate the synchronous invasion process of $m$ species in population dynamics \cite{shi}.

\section{Generalized Upper and Lower Solutions}
\noindent

In this section, we establish the existence of positive solutions of \eqref{4}. The definition of generalized upper and lower solutions is given as follows.
\begin{definition}\label{de2}
{\rm $\overline{\Phi}(\xi)=(\overline{\phi}_1(\xi), \overline{\phi}_2(\xi), \cdots, \overline{\phi}_m(\xi))$ and $\underline{\Phi}(\xi)=(\underline{\phi}_1(\xi), \underline{\phi}_2(\xi), \cdots, \underline{\phi}_m(\xi))\in C_{[\mathbf{0}, \mathbf{M}]}$ are a pair of generalized upper and lower solutions of \eqref{4} if
\begin{equation}\label{7}
\overline{\phi}_i(\xi)\ge \int_{\mathbb{R}}P_i[\psi_{11}(s-\tau c+c), \cdots, \psi_{1\tau}(s), \psi_{21}(s-\tau c+c),\cdots, \psi_{m\tau}(s)]k_i (\xi- s-c)ds,
\end{equation}
and
\begin{equation}\label{8}
\underline{\phi}_i(\xi)\le \int_{\mathbb{R}}P_i[\psi_{11}(s-\tau c+c), \cdots, \psi_{1\tau}(s), \psi_{21}(s-\tau c+c),\cdots, \psi_{m\tau}(s)]k_i (\xi- s-c)ds
\end{equation}
for all $i\in I, \xi\in \mathbb{R}$ and any uniformly continuous functions  $\psi_{ij} (\xi)$ satisfying
\[
\underline{\phi}_i(\xi)\le \psi_{ij} (\xi) \le \overline{\phi}_i(\xi) ,i\in I,  j\in J,\xi\in\mathbb{R}.
\]
}
\end{definition}
\begin{remark}{\rm
The definition implies that $\overline{\Phi}(\xi) \ge \underline{\Phi}(\xi),\xi\in\mathbb{R}.$}
\end{remark}

If $0\le \psi_{ij} (\xi) \le M_i$ is continuous for $\xi\in\mathbb{R}, i\in I, j\in J,$ then
\[
Q_i (\xi)=\int_{\mathbb{R}}P_i[\psi_{11}(s-\tau c+c), \cdots, \psi_{1\tau}(s), \psi_{21}(s-\tau c+c),\cdots, \psi_{m\tau}(s)]k_i (\xi- s-c)ds
\]
is uniformly continuous in $\xi\in\mathbb{R}$ by (k1). Moreover, $Q_i(\xi)$ admits the following property.
\begin{lemma}\label{le1}
Assume that $\underline{\phi}_i(\xi)\le \psi_{ij} (\xi)\le \overline{\phi}_i(\xi)$ for $i\in I, j\in J, \xi\in\mathbb{R}.$ Then $\underline{\phi}_i(\xi)\le Q_i (\xi)\le \overline{\phi}_i(\xi)$ for $i\in I,  \xi\in\mathbb{R}.$
\end{lemma}

The lemma is clear by Definition \ref{de2}, and we omit the proof.

Let
\[
\Gamma =\{\Phi(\xi)\in C: \underline{\Phi}(\xi)\le \Phi(\xi)\le \overline{\Phi}(\xi),\xi\in\mathbb{R} \},
\]
and define
\[
F(\Phi)(\xi)=(F_1(\Phi)(\xi), F_2(\Phi)(\xi),  \cdots, F_m(\Phi)(\xi))
\]
by
\[
F_i(\Phi)(\xi)=\int_{\mathbb{R}}P_i[\phi_{1}(s-\tau c+c), \cdots, \phi_{1}(s), \phi_{2}(s-\tau c+c),\cdots, \phi_{m}(s)]k_i (\xi- s-c)ds
\]
with $\Phi=(\phi_1,\phi_2, \cdots, \phi_m)\in \Gamma.$

\begin{lemma}\label{le2}
Assume that  $\overline{\Phi}(\xi), \underline{\Phi}(\xi)$ are a pair of generalized upper and lower solutions of \eqref{4}.
Then $F: \Gamma  \to \Gamma $ is completely continuous with respect to the decay norm $|\cdot|_{\mu}$.
\end{lemma}
\begin{proof}
According to Lemma \ref{le1}, $F: \Gamma \to \Gamma $ is true. We now prove that the mapping is completely continuous. Let
\[
\Phi =(\phi _{1},\phi _{2},\cdots ,\phi _{m})\in \Gamma ,\Psi =(\psi
_{1},\psi _{2},\cdots ,\psi _{m})\in \Gamma ,
\]
then (P2) implies that
\begin{eqnarray*}
&&\left\vert F_{i}(\Phi )(\xi )-F_{i}(\Psi )(\xi )\right\vert  \\
&=&\left\vert \int_{\mathbb{R}}P_{i}[\phi _{1}(s-\tau c+c),\cdots ,\phi
_{1}(s),\phi _{2}(s-\tau c+c),\cdots ,\phi _{m}(s)]k_{i}(\xi -s-c)ds\right.  \\
&&\left. -\int_{\mathbb{R}}P_{i}[\psi _{1}(s-\tau c+c),\cdots ,\psi _{1}(s),\psi
_{2}(s-\tau c+c),\cdots ,\psi _{m}(s)]k_{i}(\xi -s-c)ds\right\vert  \\
&\leq &L\sum_{t\in I,j\in J}\int_{\mathbb{R}}\left\vert \phi
_{t}(s-jc+c)-\psi _{t}(s-jc+c)\right\vert k_{i}(\xi -s-c)ds
\end{eqnarray*}%
and
\begin{eqnarray*}
&&\left\vert F_{i}(\Phi )(\xi )-F_{i}(\Psi )(\xi )\right\vert e^{-\mu |\xi |}
\\
&=&e^{-\mu |\xi |}\left\vert \int_{\mathbb{R}}P_{i}[\phi _{1}(s-\tau
c+c),\cdots ,\phi _{1}(s),\phi _{2}(s-\tau c+c),\cdots ,\phi
_{m}(s)]k_{i}(\xi -s-c)ds\right.  \\
&&\left. -\int_{\mathbb{R}}P_{i}[\psi _{1}(s-\tau c+c),\cdots ,\psi
_{1}(s),\psi _{2}(s-\tau c+c),\cdots ,\psi _{m}(s)]k_{i}(\xi
-s-c)ds\right\vert  \\
&\leq &Le^{-\mu |\xi |}\sum_{t\in I,j\in J}\int_{\mathbb{R}}\left\vert \phi
_{t}(s-jc+c)-\psi _{t}(s-jc+c)\right\vert k_{i}(\xi -s-c)ds.
\end{eqnarray*}%
Note that%
\begin{eqnarray*}
&&e^{-\mu |\xi |}\int_{\mathbb{R}}\left\vert \phi _{t}(s-jc+c)-\psi
_{t}(s-jc+c)\right\vert k_{i}(\xi -s-c)ds \\
&=&e^{-\mu |\xi |}\int_{\mathbb{R}}\left\vert \phi _{t}(s-jc+c)-\psi
_{t}(s-jc+c)\right\vert e^{-\mu \left\vert s-jc+c\right\vert }e^{\mu
\left\vert s-jc+c\right\vert }k_{i}(\xi -s-c)ds \\
&\leq &\left\vert \Phi -\Psi \right\vert _{\mu }\int_{\mathbb{R}}e^{-\mu
|\xi |}e^{\mu \left\vert s-jc+c\right\vert }k_{i}(\xi -s-c)ds \\
&=&\left\vert \Phi -\Psi \right\vert _{\mu }\int_{\mathbb{R}}e^{-\mu |\xi
|}e^{\mu \left\vert s-jc+c\right\vert }e^{-\mu \left\vert \xi
-s-c\right\vert }e^{\mu \left\vert \xi -s-c\right\vert }k_{i}(\xi -s-c)ds \\
&\leq &\left\vert \Phi -\Psi \right\vert _{\mu }e^{\mu jc}\int_{\mathbb{R}%
}e^{\mu \left\vert \xi -s-c\right\vert }k_{i}(\xi -s-c)ds \\
&=&\left\vert \Phi -\Psi \right\vert _{\mu }e^{\mu jc}\int_{\mathbb{R}%
}e^{\mu |\xi |}k_{i}(\xi )d\xi
\end{eqnarray*}%
because 
\[-\mu |\xi |+\mu \left\vert s-jc+c\right\vert -\mu \left\vert \xi
-s-c\right\vert \leq -\mu \left\vert s+c\right\vert +\mu \left\vert
s-jc+c\right\vert \leq \mu jc.\]
Then%
\begin{eqnarray*}
\left\vert F_{i}(\Phi )(\xi )-F_{i}(\Psi )(\xi )\right\vert e^{-\mu |\xi |}
&\leq &Lm\left\vert \Phi -\Psi \right\vert _{\mu }\sum_{j\in J}\int_{\mathbb{%
R}}e^{\mu jc}e^{\mu |\xi |}k_{i}(\xi )d\xi  \\
&\leq &Lm\tau e^{\mu \tau c}\left\vert \Phi -\Psi \right\vert _{\mu }\int_{%
\mathbb{R}}e^{\mu |\xi |}k_{i}(\xi )d\xi,
\end{eqnarray*}
which further implies that%
\[
\sup_{\xi \in \mathbb{R}}\left\vert F_{i}(\Phi )(\xi )-F_{i}(\Psi )(\xi
)\right\vert e^{-\mu |\xi |}\leq Lm \tau e^{\mu \tau c}\left\vert \Phi -\Psi \right\vert _{\mu }\int_{\mathbb{R%
}}e^{\mu |\xi | }k_{i}(\xi )d\xi
\]%
and the continuous is proved.

Moreover, if $\Phi\in \Gamma,$ then
\begin{eqnarray*}
\left\vert F_{i}(\Phi )(\xi_1 )-F_{i}(\Phi )(\xi _2)\right\vert \le M_i \int_{\mathbb{R}}\left\vert  k_{i}(\xi_1)-k_{i}(\xi_2)\right\vert ds
\end{eqnarray*}
and the equicontinuous is true by (k1). For any $\epsilon >0,$ let $B>0$ such that
\begin{equation}\label{net}
\sum_{i\in I}M_ie^{-\mu |\xi |} < \epsilon, |\xi| > B.
\end{equation}
Due to the equicontinuous, then Ascoli-Arzela lemma implies that there exist $T\in \mathbb{N}$ and
\[
\{\Phi^{i} (\xi)\}_{i=1}^{T} \subset  \{ F(\Phi)(\xi): \Phi(\xi)\in \Gamma\}
\]
such that $\{\Phi^{i} (\xi)\}_{i=1}^{T}$ is a finite $\epsilon-$net of $\{ F(\Phi)(\xi): \Phi(\xi)\in \Gamma\}$ when $|\xi|\le B$. From \eqref{net}, $\{\Phi^{i} (\xi)\}_{i=1}^{T}$ is a finite $\epsilon-$net of $\{ F(\Phi)(\xi): \Phi(\xi)\in \Gamma\}$ for $\xi\in \mathbb{R}$, which implies the compactness.
 The proof is complete.
\end{proof}

\begin{theorem}\label{th1}
Assume that  $\overline{\Phi}(\xi), \underline{\Phi}(\xi)$ are a pair of generalized upper and lower solutions of \eqref{4}.
Then \eqref{4} admits a positive solution $\Phi(\xi)$ such that
\[
\underline{\Phi}(\xi)\le \Phi (\xi)\le \overline{\Phi}(\xi).
\]
\end{theorem}
\begin{proof}
Since $\overline{\Phi}(\xi), \underline{\Phi}(\xi)$ are a pair of  generalized upper and lower solutions of \eqref{4}, then $\Gamma$ is nonempty and convex. Moreover, it is bounded and closed with respect to the decay norm $|\cdot|_{\mu}.$ By Schauder fixed point theorem and Lemma \ref{le2}, $F$ has a fixed point
\[
({\phi}^*_1(\xi), {\phi}^*_2(\xi),  \cdots, {\phi}^*_m(\xi))\in \Gamma,
\]
which is a  solution of \eqref{4}.
The proof is complete.
\end{proof}

\section{Asymptotic Behavior via Contracting Rectangles}
\noindent

In this section, we investigate the asymptotic behavior of traveling wave solutions. For the purpose, we first study the long time behavior of the following difference system
\begin{equation}\label{diff}
u_{n+1}^i=P_i[u_{n-\tau+1}^1, \cdots, u_{n}^1, u_{n-\tau+1}^2,\cdots, u_{n}^m],
\end{equation}
in which $P_i$ satisfies (P1)-(P4) for $i\in I$.
\begin{definition}\label{squ}
{\rm
Let $\textbf{R}(s),\textbf{T}(s)$ be two vector functions of $s\in [0,1]$ and take the form as follows
\[
\textbf{R}(s)=(r_1(s),r_2(s),\cdots, r_m(s))\in \mathbb{R}^m, \,\, \textbf{T}(s)=(t_1(s),t_2(s),\cdots, t_m(s))\in \mathbb{R}^m.
\]
Then $[\textbf{R}(s), \textbf{T}(s)]$ is a contracting rectangle of \eqref{diff} if
\begin{description}
\item[(C1)] $r_i(s),t_i(s)$ are continuous in $s\in [0,1],i\in I;$
\item[(C2)] $r_i(s)$ is strictly increasing in $s$ while $t_i(s)$ is strictly decreasing in $s\in [0,1],i\in I$;
\item[(C3)] $\textbf{0}\le \textbf{R}(0)< \textbf{R}(1)=\textbf{E}= \textbf{T}(1)< \textbf{T}(0)\le \textbf{M};$
\item[(C4)] for each $s\in (0,1)$ and $i, j \in I, l\in J,$
\[
r_i(s)< P_i[u_{1}^1, \cdots, u_{\tau}^1, u_{1}^2,\cdots, u_{\tau}^m]< t_i(s)
\]
for any $u^{j}_{l}\in [r_j(s),t_j(s)].$
\end{description}
}
\end{definition}

\begin{theorem}\label{th2}
Assume that $[\textbf{R}(s), \textbf{T}(s)]$ is a contracting rectangle of \eqref{diff} and there exists $s_0\in (0,1)$ such that
\begin{equation}\label{re}
r_i (s_0)\le u^i_{n'-j}\le t_i(s_0)
\end{equation}
for some $n'\in \mathbb{N}$ and all $i\in I, j\in J.$ Then
$\lim_{n\to + \infty} u^i_n= E_i, i\in I.$
\end{theorem}
\begin{proof}
By the definition of contracting rectangle, \eqref{re} implies that
\[
r_i (s_0)\le u^i_{n}\le t_i(s_0), n\ge n'.
\]
Define $\overline{u}_i, \underline{u}_i$ by
\[
\liminf_{n\to + \infty} u_n^i= \underline{u}_i, \limsup_{n\to + \infty} u_n^i= \overline{u}_i, i\in I.
\]
Were the statement false, then there exist $s_1\in [s_0,1)$ and $i_0\in I$ such that
\[
\underline{u}_{i_0}=r_{i_0}(s_1)\text{ or }\overline{u}_{i_0}=t_{i_0} (s_1)
\]
and
\[
r_i(s_1)\le \underline{u}_i \le \overline{u}_i \le t_i(s_1), i\in I.
\]
Without loss of generality, we assume that $\underline{u}_1=r_1(s_1).$ By the definition of $\liminf$ and (P2),  we obtain
\[
r_1(s_1)=P_1 [u_{1}^1, \cdots, u_{\tau}^1, u_{1}^2,\cdots, u_{\tau}^m]
\]
with some $u^{l}_{j}\in [r_l(s_1),t_l(s_1)],l\in I, j\in J,$ which contradicts the definition of contracting rectangle. The proof is complete.
\end{proof}

We now give the main result of this section.
\begin{theorem}\label{th3}
Assume that $\Phi  =(\phi_1,\cdots, \phi_m)$ is a positive solution of \eqref{4} and $[\textbf{R}(s),\textbf{T}(s)]$ is a contracting rectangle of \eqref{diff}. If
\[
\textbf{R}(s_1)\le \liminf_{\xi\to + \infty}\Phi (\xi)\le \limsup_{\xi\to + \infty}\Phi (\xi)\le \textbf{T}(s_1)
\]
with some $s_1 \in (0,1),$ then $\lim_{\xi\to + \infty} \Phi(\xi)= \textbf{E}.$
\end{theorem}
\begin{proof}
The proof is similar to that of Theorem \ref{th2}. Were the statement false, let
\[
\liminf_{\xi\to + \infty}\phi_i(\xi)=\underline{\phi}_i , \limsup_{\xi\to + \infty}\phi_i(\xi)=\overline{\phi}_i, i\in I.
\]
Without loss of generality, we assume that there exists $s_2\in [s_1,1)$ such that
\[
\underline{\phi}_1=r_1(s_2)
\]
and
\[
r_i(s_2)\le \underline{\phi}_i , \overline{\phi}_i\le t_i(s_2), i\in I.
\]
By $\liminf,$ there exists $\{\xi_t\}_{ t\in \mathbb{N}}$ with $\xi_t\to + \infty, t\to + \infty$ such that
\[
\phi_1(\xi_t)\to \underline{\phi}_1=r_1(s_2), t\to + \infty.
\]
From the dominated convergence theorem (let $t\to + \infty$ in $F$) and properties of continuous functions on bounded closed interval (see (P1)-(P2)), we obtain
\[
\underline{\phi}_1=P_1 [u_{1}^1, \cdots, u_{\tau}^1, u_{1}^2,\cdots, u_{\tau}^m]
\]
with some $u^{l}_{j}\in [r_l(s_2),t_l(s_2)],l\in I, j\in J.$ By the definition of contracting rectangles, we obtain a contradiction and complete the proof.
\end{proof}

\section{Applications}
\noindent

In this section, we shall study the traveling wave solutions of several models. To use our conclusions, we first introduce some classical theory of asymptotic spreading. Consider
\begin{equation}\label{as}
\begin{cases}
v_{n+1}(x)=\int_{\mathbb{R}} b(v_n(y))k(x-y)dy,x\in \mathbb{R}, n=0,1,2, \cdots,\\
v_0(x)=\varphi (x), x\in \mathbb{R},
\end{cases}
\end{equation}
in which $k$ satisfies (k1)-(k3) and
\begin{description}
\item[(b1)] $b(0)=0, b(v^*)=v^*$ for some $v^* >0$, and there exists $\overline{v}\ge v^*$ such that $[0, \overline{v}]$ is an invariant region of $v_{n+1}=b(v_n)$;
\item[(b2)] $b(v)$ is continuous in $v\in [0,\overline{v}]$ and $\lim_{v\to 0+} \frac{b(v)}{v}=b'(0)>1;$
\item[(b3)] $b(v) >v, v\in (0, v^*);$
\item[(b4)] $0< b(v)<  b'(0)v, v\in (0,\overline{v}];$
\item[(b5)] there exists $L_1>0$ such that
$
0\le  b'(0)v- b(v)\le  L_1 v^2, v\in [0, \overline{v}].
$
\end{description}
Moreover, from (b1)-(b5), we  define
\[
\overline{b}(v)=\sup_{u\in (0,v)}b(u), \underline{b}(v)=\inf_{u\in (v, \overline{v})}b(u), v\in [0, \overline{v}].
\]
Then both $\overline{b}(v)$ and $\underline{b}(v)$ are continuous and nondecreasing for $v\in [0, \overline{v}]$ and there exist $v_1,v_2$ such that $0< v_1\le v_2 \le \overline{v}$ and
\[
\overline{b}(v_2)=v_2, \underline{b}(v_1) =v_1.
\]

By Hsu and Zhao \cite{hsuzhao}, the following result is true.
\begin{lemma}\label{le3}
Let $c_1=: \inf_{\lambda >0}\frac{\ln \left(b^{\prime }(0)\int_{\mathbb{R}}e^{\lambda y}k(y)dy\right)}{\lambda},$ then $c_1>0.$
Assume that $0< \varphi(x)\le \overline{v}$ holds and $\varphi(x)$ is uniformly  continuous in $x\in\mathbb{R}$. Then \eqref{as} is well defined for all $n\in \mathbb{N}, x\in\mathbb{R}$ such that
$
0\le v_n(x) \le \overline{v}, n\in \mathbb{N}, x\in\mathbb{R}
$
and
\[
v_1 \le \liminf_{n\to + \infty}\inf_{|x|<cn} v_n(x) \le \limsup_{n\to + \infty} \sup_{|x|<cn}v_n(x) \le v_2\text{ for any given } c\in (0,c_1).
\]
Moreover, if $b(v)$ is nondecreasing in $v \in [0, \overline{v}]$ and $u_n(x)$ satisfies
\begin{equation*}
\begin{cases}
u_{n+1}(x)\ge \int_{\mathbb{R}} b(u_n(y))k(x-y)dy,x\in \mathbb{R}, n=0,1,2, \cdots,\\
u_0(x)\ge \varphi (x), x\in \mathbb{R},\\
0\le u_n(x)\le \overline{v}, x\in \mathbb{R}, n=0,1,2, \cdots,
\end{cases}
\end{equation*}
then
$
u_n(x)\ge v_n(x),x\in \mathbb{R}, n=0,1,2, \cdots .
$
\end{lemma}

\subsection{A Scalar Equation}
\noindent

In this part, we consider the traveling wave solutions of \eqref{2} satisfying (b1)-(b5). 
Let $v_n(x)=\varphi(x+cn)$ be a traveling wave solution of \eqref{2}, then
\begin{equation}
\varphi (\xi )=\int_{\mathbb{R}}b(\varphi (\xi -c+y))k(y)dy,\xi \in \mathbb{R%
}.  \label{5.1}
\end{equation}%
To continue our discussion, we need some constants and define
\[
\Delta(\lambda ,c)=b^{\prime }(0)\int_{\mathbb{R}}e^{\lambda y-\lambda
c}k(y)dy
\]%
for $\lambda >0, c>0.$

\begin{lemma}
If $c>c_1,$ then $\Delta(\lambda, c)=1$ has two real
positive roots $\lambda_1^c< \lambda_2^c$ such that
$\Delta(\lambda, c)<1$ for all $\lambda \in (\lambda_1^c, \lambda_2^c).$
Moreover, if $c< c_1,$ then $\Delta(\lambda, c)>1$ for all $\lambda >0.$
\end{lemma}

The lemma can be proved by the properties of convex functions, and we omit it here (see Liang and Zhao \cite[Lemma 3.8]{liang}). By these constants, we state the following results on the existence of positive traveling wave solutions of \eqref{2}.

\begin{theorem}\label{th4}
Assume that $c\geq c_{1}$ holds. Then \eqref{5.1} has a
positive solution $\varphi (\xi )$ such that $\lim_{\xi \rightarrow -\infty
}\varphi (\xi )=0$ and
\begin{equation}
v_{1}\leq \liminf_{\xi \rightarrow +\infty }\varphi (\xi )\leq \limsup_{\xi
\rightarrow +\infty }\varphi (\xi )\leq v_{2}.  \label{5.2}
\end{equation}
\end{theorem}
\begin{proof}
Let $c>c_{1}$ be fixed. Define
\[
\overline{\varphi }(\xi )=\min \{e^{\lambda _{1}^{c}\xi },v_{2}\},\underline{%
\varphi }(\xi )=\max \{e^{\lambda _{1}^{c}\xi }-qe^{\eta \lambda _{1}^{c}\xi
},0\}
\]%
with
\[
\eta \in \left( 1,\min \left\{ 2,\frac{\lambda _{2}^{c}}{\lambda _{1}^{c}%
}\right\} \right), q=1+\frac{L_1\Delta (2\lambda _{1}^{c},c)}{1-\Delta (\eta\lambda _{1}^{c},c)}.
\]%
Then we can verify that
\begin{eqnarray*}
\overline{\varphi }(\xi ) &\geq &\int_{\mathbb{R}}b(\varphi (\xi
-c+y))k(y)dy, \\
\underline{\varphi }(\xi ) &\leq &\int_{\mathbb{R}}b(\varphi (\xi
-c+y))k(y)dy
\end{eqnarray*}%
for all $\xi \in \mathbb{R}$ and any uniformly continuous functions $\varphi (\xi )$
satisfying%
\[
\overline{\varphi }(\xi )\geq \varphi (\xi )\geq \underline{\varphi }(\xi
),\xi \in \mathbb{R}.
\]%
Thus, $\overline{\varphi }(\xi )$ and $\underline{\varphi }(\xi )$ are
a pair of generalized upper and lower solutions of \eqref{5.1} and Theorem \ref{th1}
implies that \eqref{5.1} has a positive solution $\varphi (\xi )$ such that $%
\lim_{\xi \rightarrow -\infty }\varphi (\xi )=0$.

Applying Lemma \ref{le3} and the invariant form of traveling wave solutions, we
further obtain \eqref{5.2}. In fact, $\varphi (\xi )=v_n(x)$ satisfies
\[
\begin{cases} v_{n+1}(x)=\int_{\mathbb{R}}
b(v_n(y))k(x-y)dy,x\in \mathbb{R}, n=0,1,2, \cdots,\\
v_0(x)=\varphi (x), x\in \mathbb{R}, \end{cases}
\]
and so
\[
\begin{cases} v_{n+1}(x)\ge \int_{\mathbb{R}}
\underline{b}(v_n(y))k(x-y)dy,x\in \mathbb{R}, n=0,1,2, \cdots,\\
v_0(x)=\varphi (x), x\in \mathbb{R}. \end{cases}
\]
Because $\underline{b}(u)$ is nondecreasing for $u\in [0, v_2],$ then Lemma %
\ref{le3} implies that
\[
\liminf_{n\to + \infty} \inf_{|x|<2c} v_n(x) \ge v_1.
\]

From the invariant form of traveling wave solutions, we further obtain that
\[
\liminf_{\xi\to + \infty}\varphi (\xi ) = \liminf_{n\to + \infty} \inf_{|x|<2c}
v_n(x) \ge v_1
\]
because of
\[
\bigcup_{n\in \mathbb{N}, |x|<2c} \{x+cn\} \supset (0,\infty).
\]
Furthermore, $\limsup_{\xi \rightarrow + \infty }\varphi (\xi )\leq v_{2}$ is
true by the definition of $\overline{\varphi} (\xi ).$ This completes the proof when $c>c_1.$

When $c=c_{1},$ we now prove the existence of positive solutions of %
\eqref{5.1} satisfying both $\lim_{\xi \rightarrow -\infty }\varphi (\xi )=0$
and \eqref{5.2}. Let $\{c_{l}\}_{l\in \mathbb{N}}$ be a strictly decreasing
sequence satisfying $\lim_{l\rightarrow + \infty }c_{l}=c_{1}.$ Then for each $%
c_{l},$ there exists $\varphi _{l}(\xi ),$ which is a solution of \eqref{5.1}
with $c=c_{l}$ and satisfies
\begin{equation}
0< \varphi _{l}(\xi )<v_{1}/2,\xi <0,\varphi _{l}(0)=v_{1}/2  \label{la}
\end{equation}
because a traveling wave solution is invariant in the sense of phase shift (for any $a\in \mathbb{R},$ if $\varphi _{l}(\xi )$ is a traveling wave solution, then $\varphi _{l}(\xi +a)$ is one).
Evidently, (k1) implies that $\{\varphi _{l}(\xi )\}_{l\in \mathbb{N}}$ is
uniformly bounded and equicontinuous in $\xi \in \mathbb{R},l\in \mathbb{N}.$
Using Ascoli-Arzela lemma, $\varphi _{l}(\xi )$ has a subsequence, still
denoted by $\varphi _{l}(\xi ),$ and there exists a continuous function $%
\varphi ^{\ast }(\xi )$ such that $\lim_{l\rightarrow \infty }\varphi
_{l}(\xi )=\varphi ^{\ast }(\xi )$ uniformly for $\xi $ in any compact
subset of $\mathbb{R},$ and the convergence is also pointwise in $\xi\in\mathbb{R}.$ Letting $l\rightarrow + \infty $ in \eqref{5.1}, we
see that $\varphi ^{\ast }(\xi )$ satisfies \eqref{5.1} with $c=c_{1}.$
Since $\varphi ^{\ast }(\xi )$ is a bounded solution of \eqref{5.1}, then $%
\varphi ^{\ast }(\xi )$ is uniformly continuous for $\xi \in \mathbb{R}.$
Moreover, \eqref{la} implies that
\begin{equation}
0\le \varphi ^{\ast }(\xi )\leq v_{1}/2,\xi <0,\varphi ^{\ast }(0)=v_{1}/2.
\label{la1}
\end{equation}

By Lemma \ref{le3} and the invariant form of traveling wave solutions, $%
\varphi^*(\xi)$ also satisfies \eqref{5.2} and the discussion is similar to
that of $c>c_1$. We now verify that $\lim_{\xi \rightarrow -\infty
}\varphi^*(\xi)=0.$ Assume by contradiction that there exists $\epsilon >0$
such that
\[
\varphi ^{\ast }(\xi _{z})>2\epsilon
\]%
for a strictly decreasing sequence $\{\xi _{z}\}_{z\in \mathbb{N}}$
satisfying
\[
\xi _{z}<0\text{ for }z\in \mathbb{N}\text{ and }\lim_{z\rightarrow + \infty
}\xi _{z}=-\infty .
\]%
By the uniform continuity, there exists $\varepsilon >0$ such that
\[
\epsilon <\varphi ^{\ast }(\xi )\leq v_{1}/2,\xi \in \lbrack \xi
_{z}-\varepsilon ,\xi _{z}+\varepsilon ],z\in \mathbb{N}.
\]

Consider the initial value problem
\[  \label{ass}
\begin{cases} w_{n+1}(x)=\int_{\mathbb{R}}
\underline{b}(w_n(y))k(x-y)dy,x\in \mathbb{R}, n=0,1,2, \cdots,\\
w_0(x)=\omega (x), x\in \mathbb{R}, \end{cases}
\]
in which $\omega (x)$ is a continuous function satisfying

\begin{description}
\item[(w1)] $\omega (x)=0, |x|\ge \varepsilon;$

\item[(w2)] $\omega (x)$ is decreasing for $x\in [\varepsilon /2,
\varepsilon];$

\item[(w3)] $\omega (x)=\omega (-x),x\in\mathbb{R}; $

\item[(w4)] $\omega (x)=\epsilon, |x|< \varepsilon /2.$
\end{description}

By Lemma \ref{le3}, there exists $T\in\mathbb{N}$ such that
\[
w_n(0)> 3 v_1/4, n\ge T.
\]
Note that $v_n(x)=\varphi^*(x+ c_1 n),$ then for each $\xi _{z},z\in\mathbb{N},$
\[  \label{ass*}
\begin{cases} v_{n+1}(x+\xi_z)\ge \int_{\mathbb{R}}
\underline{b}(v_n(y))k(x+\xi_z-y)dy,x\in \mathbb{R}, n=0,1,2, \cdots,\\ v_0(x+\xi_z)\ge
\omega (x), x\in \mathbb{R}. \end{cases}
\]
From Lemma \ref{le3}, we see that
\[
v_n(\xi _{z}) \ge w_n(0), z\in\mathbb{N}, n\in\mathbb{N},
\]
and
\[
v_T(\xi _{z})=\varphi ^{\ast }(\xi _{z}+c_{1}T)>3v_{1}/4,z\in \mathbb{N}.
\]%
Let $z>0$ such that $\xi _{z}+c_{1}T<0,$ then \eqref{la1} leads to
\[
\varphi ^{\ast }(\xi _{z}+c_1T)\leq v_{1}/2,
\]%
which implies a contradiction. Thus, $\lim_{\xi\to -\infty}\varphi^*(\xi)=0.$
The proof is complete.
\end{proof}
\begin{remark}{\rm
When $c=c_1,$ Hsu and Zhao \cite[Theorem 3.2]{hsuzhao} proved the existence of $\varphi (\xi )$ satisfying \eqref{5.1}-\eqref{5.2}, but they did not answer the existence of $\lim_{\xi \rightarrow -\infty
}\varphi (\xi ).$ Moreover, for $c<c_1 (c>c_1),$ the nonexistence (existence) of traveling wave solutions of \eqref{as} has been proved by Hsu and Zhao \cite{hsuzhao}.}
\end{remark}

If $b(v)$ is monotone for $v\in \lbrack 0,v^{\ast }],$ then the limit
behavior of traveling wave solutions has been obtained by the monotonicity of traveling wave solutions \cite{hsuzhao}. If $b(v)$ is not monotone for $v\in
\lbrack 0,v^{\ast }],$ then the limit behavior should be further
investigated, and Hsu and Zhao \cite{hsuzhao} gave some sufficient
conditions on the topic. Using the contracting rectangle, we may prove some
results on the limit behavior although the verification is technical. For
example, let
\begin{equation}
b(v)=3v(1-v),  \label{5.3}
\end{equation}%
then Hsu and Zhao \cite{hsuzhao} has shown that a positive traveling wave solution of %
\eqref{5.1} with \eqref{5.3} satisfies $\lim_{\xi \rightarrow + \infty }{\varphi }(\xi )=\frac{2%
}{3}.$ We now construct a proper contracting rectangle to illustrate our conclusion.

For $b(v)=3v(1-v),$ we take
\[
\overline{v}=v_2= \frac{3}{4}, v_1=\frac{9}{16}, v^*=\frac{2}{3}.
\]
By Theorem \ref{th4}, we obtain $0\le \varphi (\xi ) \le \frac{3}{4},\xi\in\mathbb{R}$ and
\[
\frac{9}{16}\leq \liminf_{\xi \rightarrow +\infty }\varphi (\xi )\leq
\limsup_{\xi \rightarrow +\infty }\varphi (\xi )\leq \frac{3}{4}.
\]
Applying the dominated convergence theorem, we see that
\[
\limsup_{\xi \rightarrow +\infty }\varphi (\xi )\le \sup_{v\in [\frac{9}{16}, \frac{3}{4}]} b(v)=\frac{189}{256}
\]
and
\[
\liminf_{\xi \rightarrow +\infty }\varphi (\xi )\ge \inf_{v\in [\frac{9}{16}, \frac{189}{256}]} b(v)=\frac{37989}{65536} \left(>\frac{9}{16}=\frac{36864}{65536}\right),
\]
then
\[
\frac{9}{16}<\liminf_{\xi \rightarrow +\infty }\varphi (\xi )\leq
\limsup_{\xi \rightarrow +\infty }\varphi (\xi )<\frac{3}{4}.
\]%
Furthermore, we  have
\begin{equation}
\begin{cases}
b^{2}(v)>v,v\in \lbrack \frac{9%
}{16},\frac{2}{3}),\\
b^{2}(v)<v,v\in (\frac{2}{3},\frac{3}{4}].
\end{cases}\label{5.4}
\end{equation}%

Let
\[
r(s)=\frac{9}{16}+\frac{5s}{48},\,\,t (s)=b(r(s))+\epsilon
(b^{2}(r(s))-r(s))
\]%
for some $\epsilon >0.$ It suffices to verify the definition of contracting rectangle.
\begin{lemma}
\label{le4} If $\epsilon >0$ is small enough, then $[r(s),t (s)]$ defines
a contracting rectangle of $v_{n+1}=3v_{n}(1-v_{n}).$
\end{lemma}

\begin{proof}
If $s\in (0,1)$ and $\epsilon >0,$ then $r(s)\in (\frac{9}{16},\frac{2}{3}).$
By \eqref{5.4}, we obtain
\[
b^{2}(r(s))-r(s)>0.
\]
Then the monotonicity of $b(v), v\in [\frac12,1]$ implies that
\[
t (s)>b(u(s)),u(s)\in \lbrack r(s), t (s)]
\]
because
$
t(s)> b(r(s)), s\in (0,1).
$

To prove that
\[
r(s)<b(u(s)),u(s)\in \lbrack r(s),t (s)],
\]%
we first assume that $\epsilon >0$ is small such that
\[
\frac{2}{3}<t (s)<\frac{3}{4},s\in (0,1).
\]%
Then the monotonicity of $b(u),u\in (\frac{1}{2},1)$ implies that we only need to
verify that
\[
r(s)<b(t(s)).
\]%
Note that $\sup_{u\in \lbrack \frac{9}{16},\frac{3}{4}]}\left\vert
3-6u\right\vert <2$ such that%
\[
\left\vert b(u)-b(v)\right\vert <2\left\vert u-v\right\vert ,u,v\in \lbrack
\frac{9}{16},\frac{3}{4}].
\]%
We further have
\begin{eqnarray*}
b(t (s))-r(s)
&=&b(b(r(s)+\epsilon (b^{2}(r(s))-r(s))))-r(s) \\
&>&b^{2}(r(s))-r(s)-2\epsilon (b^{2}(r(s))-r(s))) \\
&=& (1-2\epsilon)[b^{2}(r(s))-r(s)]\\
&>&0,s\in (0,1)
\end{eqnarray*}%
if $\epsilon \in (0, \frac{1}{2})$. The proof is complete.
\end{proof}

\subsection{A Competitive System of Two Species}
\noindent

In this part, we consider the traveling wave solutions of the following system
\begin{equation}\label{5.6}
\begin{cases}
p_{n+1}(x)=\int_{\mathbb{R}}
\frac{(1+d_1)p_n(x-y)}{1+d_1(p_n(x-y)+b_1p_{n-1}(x-y)+a_1q_n(x-y))}k_1(y)dy,\\
q_{n+1}(x)=\int_{\mathbb{R}}
\frac{(1+d_2)q_n(x-y)}{1+d_2(q_n(x-y)+b_2q_{n-1}(x-y)+a_2p_n(x-y))}k_2(y)dy,
\end{cases}
\end{equation}
in which $n=0,1,2,\cdots,  x\in \mathbb{R},$ all the parameters are nonnegative,
and $k_1(y),k_2(y)$ are probability functions describing the
migration of the individuals and satisfy (k1)-(k3). In particular, we assume that
\begin{equation}\label{5.7}
d_1>0,\, d_2>0, \, 1+b_1>a_2, \, 1+b_2>a_1
\end{equation}
such that \eqref{5.6} admits a spatially homogeneous steady state
\[
(p^*, q^*) =\left(\frac{1+b_2-a_1}{(1+b_1)(1+b_2)-a_1a_2} , \frac{1+b_1-a_2}{(1+b_1)(1+b_2)-a_1a_2}\right)\gg (0,0)
\]
with
\[
\begin{cases}
p^{\ast }(1+b_{1})+a_{1}q^{\ast }=1,\\
q^{\ast }(1+b_{2})+a_{2}p^{\ast }=1.
\end{cases}
\]

If $b_1=b_2=0,$ then \eqref{5.6} is the model in Lewis et al. \cite{lewis}, Lin et al. \cite{llrjmb}, Zhang and Zhao \cite{zhangzhao}. More precisely, if $0<a_1<1<a_2,$ then Lewis et al. \cite{lewis} investigated the propagation modes between a resident and an invader. When $a_1>1, a_2>1,$ Zhang and Zhao \cite{zhangzhao} proved the existence and stability of bistable traveling wave solutions. Recently,
Lin and Li \cite{linlidc} and Lin et al. \cite{llrjmb} obtained the existence of traveling wave solutions describing the coinvasion of two competitors if $k_i$ is the Gaussian or admits compact support and $a_1,a_2\in [0,1)$. In this paper, we assume that $k_1(y),k_2(y)$ satisfy (k1)-(k3) and establish the existence/nonexistence of traveling wave solutions, which covers and completes the corresponding results in \cite{linlidc,llrjmb}.

Let
$
\rho (\xi)= p_n(x), \varrho (\xi)=q_n(x), \xi=x+cn
$
be a traveling wave solution of \eqref{5.6}, then
\begin{equation}\label{5.8}
\begin{cases}
\rho (\xi +c)=\int_{\mathbb{R}}
\frac{(1+d_1)\rho (\xi -y)}{1+d_1(\rho (\xi -y)+b_1\rho (\xi -y-c)+a_1\varrho (\xi-y))}k_1(y)dy,\\
 \varrho (\xi+c)=\int_{\mathbb{R}}
\frac{(1+d_2)\varrho (\xi-y)}{1+d_2(\varrho (\xi-y)+b_2\varrho (\xi-y-c)+a_2\rho (\xi -y))}k_2(y)dy.
\end{cases}
\end{equation}

To study \eqref{5.8}, for $\lambda >0, c>0,$ we also define
\[
\Theta_1(\lambda, c)=(1+d_1)\int_{\mathbb{R}}e^{\lambda y- \lambda c} k_1(y)dy, \Theta_2(\lambda, c)=(1+d_2)\int_{\mathbb{R}}e^{\lambda y- \lambda c} k_2(y)dy.
\]
\begin{lemma}\label{le6}
Let
\[
c^*= \max\left\{ \inf_{\lambda >0}\frac{\ln \left((1+d_1)\int_{\mathbb{R}}e^{\lambda y}k_1(y)dy\right)}{\lambda}, \inf_{\lambda >0}\frac{\ln \left((1+d_2)\int_{\mathbb{R}}e^{\lambda y}k_2(y)dy\right)}{\lambda}\right\}.
\]
Then $c^*>0$ holds and $
\Theta_i(\lambda,c)=1 $ has  two distinct positive roots
$\lambda_{i1}^c< \lambda_{i2}^c$ for any $c>c^*$ and each $i=1,2$.
Moreover, if $c\in (0, c^*),$ then $\Theta_1(\lambda,c)>1
$ for any $\lambda \ge 0$ or $\Theta_2(\lambda,c)> 1$ for any $\lambda \ge 0$. In
addition, for any given $c> c^*,$ there exists $\eta \in (1,2)$ such
that $\eta\lambda_{i1}^c< \lambda_{11}^c+\lambda_{21}^c$ and $\Theta_i(\lambda,c)<1 $ for all $\lambda \in
(\lambda_{i1}^c, \eta \lambda_{i1}^c],$ $ i=1,2.$
\end{lemma}
\begin{theorem}\label{th5}
For each $c>c^*,$ \eqref{5.8} has a positive solution $(\rho (\xi), \varrho (\xi))$ satisfying
\[
0< \rho (\xi) <1, 0< \varrho (\xi) <1, \xi\in\mathbb{R}.
\]
\end{theorem}
\begin{proof}
For $q>1,$ define continuous functions
\[
\overline{\rho }(\xi )=\min \{e^{\lambda _{11}^{c}\xi },1\},\overline{%
\varrho }(\xi )=\min \{e^{\lambda _{21}^{c}\xi },1\}
\]%
and
\[
\underline{\rho }(\xi )=\max \{e^{\lambda _{11}^{c}\xi }-qe^{\eta \lambda
_{11}^{c}\xi },0\},\underline{\varrho }(\xi )=\max \{e^{\lambda _{21}^{c}\xi
}-qe^{\eta \lambda _{21}^{c}\xi },0\}
\]%
with $\eta $ formulated by Lemma \ref{le6}. We now prove that $(\overline{\rho }(\xi ),\overline{\varrho }(\xi )),(%
\underline{\rho }(\xi ),\underline{\varrho }(\xi ))$ are a pair if
generalized upper and lower solution of \eqref{5.8} for any uniformly
continuous functions
\[
\rho _{1}(\xi ),\rho _{2}(\xi ),\varrho _{1}(\xi ),\varrho _{2}(\xi )
\]%
satisfying
\[
\underline{\rho }(\xi )\leq \rho _{1}(\xi ),\rho _{2}(\xi )\leq \overline{%
\rho }(\xi ),\underline{\varrho }(\xi )\leq \varrho _{1}(\xi ),\varrho
_{2}(\xi )\leq \overline{\varrho }(\xi ),\xi \in \mathbb{R}.
\]%
Namely, for $\overline{\rho }(\xi )$ and $\underline{\rho }(\xi ),$ we need
to verify that%
\[
\underline{\rho }(\xi +c)\leq \int_{\mathbb{R}}\frac{(1+d_{1})\rho _{1}(\xi
-y)}{1+d_{1}(\rho _{1}(\xi -y)+b_{1}\rho _{2}(\xi -y-c)+a_{1}\varrho
_{1}(\xi -y))}k_{1}(y)dy\leq \overline{\rho }(\xi +c)
\]%
for any $\xi \in \mathbb{R}.$ From the monotonicity of%
\[
\frac{(1+d_{1})u}{1+d_{1}(u+b_{1}v+a_{1}w)},u,v,w\geq 0,
\]%
it suffices to verify that%
\begin{equation}
\underline{\rho }(\xi +c)\leq \int_{\mathbb{R}}\frac{(1+d_{1})\underline{%
\rho }(\xi -y)}{1+d_{1}(\underline{\rho }(\xi -y)+b_{1}\overline{\rho }(\xi
-y-c)+a_{1}\overline{\varrho }(\xi -y))}k_{1}(y)dy,\xi \in \mathbb{R},
\label{sub}
\end{equation}%
and%
\begin{equation}
\int_{\mathbb{R}}\frac{(1+d_{1})\overline{\rho }(\xi -y)}{1+d_{1}\overline{%
\rho }(\xi -y)}k_{1}(y)dy\leq \overline{\rho }(\xi +c),\xi \in \mathbb{R}.
\label{sup}
\end{equation}

If $\overline{\rho }(\xi +c)=e^{\lambda _{11}^{c}(\xi +c)},$ then
\begin{eqnarray*}
\int_{\mathbb{R}}\frac{(1+d_{1})\overline{\rho }(\xi -y)}{1+d_{1}\overline{%
\rho }(\xi -y)}k_{1}(y)dy
&\leq &\int_{\mathbb{R}}(1+d_{1})\overline{\rho }(\xi -y)k_{1}(y)dy \\
&\leq &\int_{\mathbb{R}}(1+d_{1})e^{\lambda _{11}^{c}(\xi -y)}k_{1}(y)dy \\
&=&e^{\lambda _{11}^{c}(\xi +c)}=\overline{\rho }(\xi +c).
\end{eqnarray*}%
If $\overline{\rho }(\xi +c)=1,$ then
\[
\int_{\mathbb{R}}\frac{(1+d_{1})\overline{\rho }(\xi -y)}{1+d_{1}\overline{%
\rho }(\xi -y)}k_{1}(y)dy\leq \int_{\mathbb{R}}\frac{(1+d_{1})\overline{\rho
}(\xi -y)}{1+d_{1}\overline{\rho }(\xi -y)}k_{1}(y)dy\leq 1=\overline{\rho }%
(\xi +c).
\]
By what we have done, we have proved that \eqref{sup} holds.

If $\underline{\rho }(\xi +c)=0,$ then
\[
\frac{(1+d_{1})\underline{\rho }(\xi -y)}{1+d_{1}(\underline{\rho }(\xi
-y)+b_{1}\overline{\rho }(\xi -y-c)+a_{1}\overline{\varrho }(\xi -y))}\geq
0,y\in \mathbb{R}
\]%
and
\[
\int_{\mathbb{R}}\frac{(1+d_{1})\underline{\rho }(\xi -y)}{1+d_{1}(%
\underline{\rho }(\xi -y)+b_{1}\overline{\rho }(\xi -y-c)+a_{1}\overline{%
\varrho }(\xi -y))}k_{1}(y)dy\geq 0=\underline{\rho }(\xi +c)
\]%
is clear. Otherwise, from
\[
\frac{1}{1+a}\geq 1-a,a\geq 0,
\]%
we see that%
\begin{eqnarray*}
&&\int_{\mathbb{R}}\frac{(1+d_{1})\underline{\rho }(\xi -y)}{1+d_{1}(%
\underline{\rho }(\xi -y)+b_{1}\overline{\rho }(\xi -y-c)+a_{1}\overline{%
\varrho }(\xi -y))}k_{1}(y)dy \\
&\geq &\int_{\mathbb{R}}(1+d_{1})\underline{\rho }(\xi -y)[1-d_{1}(%
\underline{\rho }(\xi -y)+b_{1}\overline{\rho }(\xi -y-c)+a_{1}\overline{%
\varrho }(\xi -y))]k_{1}(y)dy \\
&=&\int_{\mathbb{R}}(1+d_{1})\underline{\rho }(\xi -y)k_{1}(y)dy \\
&&-d_{1}(1+d_{1})\int_{\mathbb{R}}\underline{\rho }(\xi -y)(\underline{\rho }%
(\xi -y)+b_{1}\overline{\rho }(\xi -y-c)+a_{1}\overline{\varrho }(\xi
-y))k_{1}(y)dy.
\end{eqnarray*}%
Because of $\underline{\rho }(s)\geq e^{\lambda _{11}^{c}s}-qe^{\eta \lambda
_{11}^{c}s},s\in \mathbb{R},$ we have%
\begin{eqnarray*}
\int_{\mathbb{R}}(1+d_{1})\underline{\rho }(\xi -y)k_{1}(y)dy
&\geq &\int_{\mathbb{R}}(1+d_{1})[e^{\lambda _{11}^{c}(\xi -y)}-qe^{\eta
\lambda _{11}^{c}(\xi -y)}]k_{1}(y)dy \\
&=&e^{\lambda _{11}^{c}(\xi +c)}-q\Theta _{1}(\eta \lambda
_{11}^{c},c)e^{\eta \lambda _{11}^{c}(\xi +c)}
\end{eqnarray*}%
by Lemma \ref{le6}. At the same time,%
\[
\overline{\rho }(s)\le \overline{\rho }(s+c),\underline{\rho }(s)\leq \overline{\rho }(s)\leq e^{\lambda _{11}^{c}s},%
\underline{\varrho }(s)\leq \overline{\varrho }(s)\leq e^{\lambda
_{21}^{c}s},s\in \mathbb{R}
\]%
such that%
\begin{eqnarray*}
&&-d_{1}(1+d_{1})\int_{\mathbb{R}}\underline{\rho }(\xi -y)(\underline{\rho }%
(\xi -y)+b_{1}\overline{\rho }(\xi -y-c)+a_{1}\overline{\varrho }(\xi
-y))k_{1}(y)dy \\
&\geq &-\int_{\mathbb{R}}d_{1}(1+d_{1})[(1+b_{1})\overline{\rho }^{2}(\xi
-y)+a_{1}\overline{\rho }(\xi -y)\overline{\varrho }(\xi -y)]k_{1}(y)dy \\
&\geq &-\int_{\mathbb{R}}d_{1}(1+d_{1})[(1+b_{1})e^{2\lambda _{11}^{c}(\xi
-y)}+a_{1}e^{(\lambda _{11}^{c}+\lambda _{21}^{c})(\xi -y)}]k_{1}(y)dy \\
&=&-d_{1}(1+b_{1})\Theta _{1}(2\lambda _{11}^{c},c)e^{2\lambda _{11}^{c}(\xi
+c)}-d_{1}a_{1}\Theta _{1}(\lambda _{11}^{c}+\lambda _{21}^{c},c)e^{(\lambda
_{11}^{c}+\lambda _{21}^{c})(\xi +c)}
\end{eqnarray*}%
by Lemma \ref{le6}. Then \eqref{sub} holds provided that
\[
q>\frac{d_{1}(1+b_{1})\Theta _{1}(2\lambda _{11}^{c},c)+d_{1}a_{1}\Theta
_{1}(\lambda _{11}^{c}+\lambda _{21}^{c},c)}{1-\Theta _{1}(\eta \lambda
_{11}^{c},c)}+1.
\]%

Similarly, if
\[
q>\frac{d_{2}(1+b_{2})\Theta _{2}(2\lambda _{21}^{c},c)+d_{2}a_{2}\Theta
_{2}(\lambda _{11}^{c}+\lambda _{21}^{c},c)}{1-\Theta _{2}(\eta \lambda
_{21}^{c},c)}+1,
\]%
then we obtain the inequalities satisfied by $\underline{\varrho }(\xi ),%
\overline{\varrho }(\xi )$ and a pair of generalized upper and lower
solutions of \eqref{5.8}.

By Theorem \ref{th1}, \eqref{5.8} has a solution $(\rho (\xi), \varrho (\xi))
$ satisfying
\[
0\le \rho (\xi) \le 1, 0\le \varrho (\xi) \le 1
\]
and
\[
\lim_{\xi\to -\infty}\rho (\xi)=\lim_{\xi\to -\infty}\varrho (\xi)=0,
\sup_{\xi\in\mathbb{R}}\rho (\xi) >0, \sup_{\xi\in\mathbb{R}}\varrho (\xi)
>0.
\]

If $\rho (\xi)=0$ for some $\xi\in \mathbb{R},$ then (k1) implies that $\rho
(\xi) \equiv 0,$ which is a contradiction. Similarly, we can obtain $0< \rho
(\xi) <1, 0< \varrho (\xi) <1$ and the proof is complete.
\end{proof}

In what follows, we consider the limit behavior of traveling wave solutions when
\begin{equation}\label{5.9}
a_1+b_1 <1, a_2+b_2 <1.
\end{equation}
\begin{lemma}\label{le5}
Assume that \eqref{5.7} and \eqref{5.9} hold and define
\begin{eqnarray*}
r_{1}(s) &=&sp^*,t_{1}(s)=sp^*+(1+\epsilon )\left( 1-s\right) ,\, \\
\,r_{2}(s) &=&sq^*,t_{2}(s)=sq^*+(1+\epsilon )\left( 1-s\right)
\end{eqnarray*}%
with $\epsilon >0$ such that
\[
(b_{1}+a_{1})(1+\epsilon )<1, (b_{2}+a_{2})(1+\epsilon )<1.
\]
Then $[r_1(s), t_1(s)]\times [r_2(s), t_2(s)]$ is a contracting rectangle of
\begin{equation}\label{5.90}
\begin{cases}
p_{n+1}=
\frac{(1+d_1)p_n}{1+d_1(p_n+b_1p_{n-1}+a_1q_n)}, \\
q_{n+1}=
\frac{(1+d_2)q_n}{1+d_2(q_n+b_2q_{n-1}+a_2p_n)}.
\end{cases}
\end{equation}
\end{lemma}
\begin{proof}
The continuity and monotonicity in these four functions are clear and we just verify (C4) of Definition \ref{squ}. Since
\[
\frac{(1+d_1)p_n}{1+d_1(p_n+b_1p_{n-1}+a_1q_n)}
\]
is increasing in $p_n\ge 0$ and decreasing in $p_{n-1}\ge 0, q_n\ge 0,$ then for the first equation of \eqref{5.90}, it suffices to prove that
\begin{equation}\label{5.10}
r_1(s)<
\frac{(1+d_1)r_1(s)}{1+d_1(r_1(s)+b_1t_1(s)+a_1 t_2(s))}
\end{equation}
and
\begin{equation}\label{5.11}
t_1(s)>
\frac{(1+d_1)t_1(s)}{1+d_1(t_1(s)+b_1r_1(s)+a_1 r_2(s))}
\end{equation}
for $s\in (0,1).$ In fact, $(b_{1}+a_{1})(1+\epsilon )<1$ and $s\in (0,1)$ imply that
\begin{eqnarray*}
&&\frac{1+d_{1}}{1+d_{1}(r_{1}(s)+b_{1}t_{1}(s)+a_{1}t_{2}(s))} \\
&=&\frac{1+d_{1}}{1+d_{1}(sp^*+b_{1}(sp^*+(1+\epsilon )\left( 1-s\right)
)+a_{1}(sq^*+(1+\epsilon )\left( 1-s\right) )} \\
&=&\frac{1+d_{1}}{1+d_{1}(s+(b_{1}+a_{1})(1+\epsilon )(1-s))} \\
&>&\frac{1+d_{1}}{1+d_{1}(s+(1-s))} \\
&=&1
\end{eqnarray*}
and \eqref{5.10} has been proved. Moreover, \eqref{5.11} is true since
\begin{eqnarray*}
\frac{1+d_{1}}{1+d_{1}(t_{1}(s)+b_{1}r_{1}(s)+a_{1}r_{2}(s))}
&=&\frac{1+d_{1}}{1+d_{1}(sp^{\ast }+(1+\epsilon )\left( 1-s\right)
+b_{1}sp^{\ast }+a_{1}sq^{\ast })} \\
&=&\frac{1+d_{1}}{1+d_{1}(s+(1+\epsilon )\left( 1-s\right) )} \\
&<&\frac{1+d_{1}}{1+d_{1}(s+1-s)} \\
&=&1.
\end{eqnarray*}

In a similar way, we have
\begin{eqnarray*}
r_{2}(s) &<&\frac{(1+d_{2})r_{2}(s)}{%
1+d_{2}(r_{2}(s)+b_{2}t_{2}(s)+a_{2}t_{1}(s))}, \\
t_{2}(s) &>&\frac{(1+d_{2})t_{2}(s)}{%
1+d_{2}(t_{2}(s)+b_{2}r_{2}(s)+a_{2}r_{1}(s))}
\end{eqnarray*}%
for $s\in (0,1).$ The proof is complete.
\end{proof}

We now give our main conclusion on the asymptotic behavior of traveling wave solutions of \eqref{5.6}.
\begin{theorem}\label{th6}
Assume that $(\rho (\xi), \varrho (\xi))$ is given by Theorem \ref{th5}. If \eqref{5.9} holds, then
\[
\lim_{\xi\to + \infty}(\rho (\xi), \varrho (\xi))=(p^*,q^*).
\]
\end{theorem}
\begin{proof}
It is clear that $p_{n}(x)=\rho (\xi )$ satisfies
\[
\begin{cases}
p_{n+1}(x)\ge \int_{\mathbb{R}}\frac{(1+d_{1})p_{n}(x-y)}{%
1+d_{1}(p_{n}(x-y)+b_{1}+a_{1})}k_{1}(y)dy,x\in \mathbb{R},n=0,1,2,\cdots
,\\
p_{0}(x)=\rho (x)>0,x\in \mathbb{R}.
\end{cases}
\]%
Consider the initial value problem%
\[
\begin{cases}
w_{n+1}(x)=\int_{\mathbb{R}}\frac{(1+d_{1})w_{n}(x-y)}{%
1+d_{1}(w_{n}(x-y)+b_{1}+a_{1})}k_{1}(y)dy,x\in \mathbb{R},n=0,1,2,\cdots
,\\
w_{0}(x)=\rho (x)>0,x\in \mathbb{R}.
\end{cases}
\]%
Because $\frac{(1+d_{1})w}{1+d_{1}(w+b_{1}+a_{1})}$ is monotone increasing
in $w\geq 0,$ then Lemma \ref{le3} implies that
\[
\liminf_{n\rightarrow +\infty }\inf_{|x|<2c}w_{n}(x)\geq 1-a_{1}-b_{1},
\]
and
\[
w_n(x)\le p_n(x), n\in\mathbb{N}, x\in \mathbb{R},
\]
which further leads to
\[
\liminf_{\xi \rightarrow +\infty }\rho (\xi )= \liminf_{n\rightarrow +\infty }\inf_{|x|<2c}p_{n}(x)\geq \liminf_{n\rightarrow +\infty }\inf_{|x|<2c}w_{n}(x)\geq 1-a_{1}-b_{1}>0
\]%
by the invariant form of traveling wave solutions.

In a similar way, we have
\[
\liminf_{\xi \rightarrow +\infty }\varrho (\xi )\geq 1-a_{2}-b_{2}.
\]

By what we have done, there exists $s_0 \in (0,1)$ such that
\[
r_1(s_0)< \liminf_{\xi\to + \infty} \rho (\xi)\le \limsup_{\xi\to + \infty} \rho
(\xi) < t_1(s_0)
\]
and
\[
r_2(s_0)< \liminf_{\xi\to + \infty} \varrho (\xi)\le \limsup_{\xi\to + \infty}
\varrho (\xi) < t_2(s_0).
\]

Using Theorem \ref{th3} and Lemma \ref{le5}, we complete the proof.
\end{proof}

In fact, any bounded positive traveling wave solutions of \eqref{5.8} admit the following nice property.
\begin{lemma}\label{le8.0}
Assume that $(\rho(\xi), \varrho(\xi))$ is a bounded positive solution  of \eqref{5.8}. Then
\[
0< \rho (\xi), \varrho (\xi) \le  1, \xi\in\mathbb{R}.
\]
\end{lemma}
\begin{proof}
Since $k_1, k_2$ are Lebesgue integrable, and
\[
\frac{(1+d_{1})u_1}{1+d_{1}(u_1+b_{1}v_1+a_{1}w_1)}=0,\frac{(1+d_{2})u_2}{%
1+d_{2}(u_2+b_{2}v_2+a_{2}w_2)}=0
\]
for $u_i,v_i,w_i\in \lbrack 0,\infty ), i=1,2,$ if and only if $u_1=u_2=0,$ then $\rho (\xi_1)=0$ for some $\xi_1$ implies that $\rho (\xi)\equiv 0$ for all $ \xi\in\mathbb{R}.$ Thus, $\rho (\xi) >0, \xi\in\mathbb{R}.$ By the dominated convergence theorem and the monotonicity,
\[
\sup_{\xi \in \mathbb{R}}\rho (\xi )\leq \sup_{\xi \in \mathbb{R}}\frac{%
(1+d_{1})\rho (\xi )}{1+d_{1}(\rho (\xi )+b_{1}\rho (\xi -c)+a_{1}\varrho
(\xi ))}\leq \frac{(1+d_{1})\sup_{\xi \in \mathbb{R}}\rho (\xi )}{%
1+d_{1}\sup_{\xi \in \mathbb{R}}\rho (\xi )},
\]
which implies $\sup_{\xi \in \mathbb{R}}\rho (\xi )\in (0,1].$

Similarly, we can obtain the boundedness of $\varrho (\xi).$ The proof is complete.
\end{proof}

Combining Theorem \ref{th6} with Lemma \ref{le8.0}, we further have the following conclusion.
\begin{corollary}
Assume that $(\rho (\xi),\varrho (\xi) )$ is a bounded positive solution of \eqref{5.8}. If \eqref{5.9} holds, then
$
\lim_{\xi\to + \infty}(\rho (\xi), \varrho (\xi))=(p^*,q^*).
$
\end{corollary}

\begin{theorem}\label{th5.}
If $c<c^*,$ then \eqref{5.8} does not have positive solutions satisfying
\begin{equation}\label{5.a}
\lim_{\xi\to -\infty}\rho (\xi)=\lim_{\xi\to -\infty}\varrho (\xi)=0, \liminf_{\xi\to + \infty}\rho (\xi) >0, \liminf_{\xi\to + \infty}\varrho (\xi) >0.
\end{equation}
\end{theorem}
\begin{proof}
Without loss of generality, we assume that
\[
c^*=  \inf_{\lambda >0}\frac{\ln \left((1+d_1)\int_{\mathbb{R}}e^{\lambda y}k_1(y)dy\right)}{\lambda}.
\]

We first give two claims as follows:
\begin{description}
\item[(D1)]  for any given $c_1<c^*,$ there exists $\epsilon_0 >0$ such that
\[
\Theta_1^{\epsilon}(\lambda, c) >1\text{ for any } \epsilon \in (0, \epsilon_0), c<c_1, \lambda >0,
\]
where
\[
\Theta_1^{\epsilon}(\lambda, c)=\frac{1+d_1}{1+d_1\epsilon(a_1+b_1)}\int_{\mathbb{R}}e^{\lambda y- \lambda c} k_1(y)dy;
\]
\item[(D2)] assume that $\rho(\xi)>0,\varrho(\xi)>0,\xi \in\mathbb{R}$ satisfy \eqref{5.a}, then  for any given $\varepsilon >0,$ there exists $M=M(\varepsilon, \rho, \varrho)>1$ such that
\[
b_1\rho (\xi -y-c)+a_1\varrho (\xi-y)\le \max\{ \varepsilon, (M-1)\rho (\xi -y)\}, \xi-y\in \mathbb{R}.
\]
\end{description}
In fact, (D1) is evident by the definition of $c^*.$ For any $\varepsilon >0,$
the limit behavior \eqref{5.a} implies that there exists $T<0$ such that
\[
b_1\rho (\xi -y-c)+a_1\varrho (\xi-y)\le \varepsilon, \xi -y \le  T + 2c.
\]
When $\xi -y \ge T,$ then
the properties of continuous functions and $\liminf_{\xi\to + \infty}\rho (\xi) >0$ indicate that there exists $\delta >0$ with
$\inf_{\xi -y \ge T}\rho (\xi-y) > \delta.$
By Lemma \ref{le8.0}, we can take $M$ satisfying
$(M-1)\delta > b_1+ a_1$
and
\[
b_1\rho (\xi -y-c)+a_1\varrho (\xi-y)\le (M-1)\rho (\xi-y), \xi -y \ge T +c,
\]
and we obtain (D2).

Were the statement false, then there exists $c_2< c^*$ such that  \eqref{5.8} with $c=c_2$ has a (fixed) positive solution $(\rho, \varrho)$ satisfying \eqref{5.a}.  Let $\epsilon_1>0$ be small such that
\[
\Theta_1^{2\epsilon}(\lambda, c) >1\text{ for any } \epsilon \in (0, \epsilon_1), c<\frac{c_2+c^*}{2}, \lambda >0.
\]
At the same time, (D2) indicates that there exists $M>1$ such that
\[
\rho (\xi +c_2)\ge \int_{\mathbb{R}}
\frac{(1+d_1)\rho (\xi -y)}{1+d_1((a_1+b_1)\epsilon_1+M\rho (\xi -y))}k_1(y)dy,
\]
which leads to
\[
p_{n+1}(x)\ge  \int_{\mathbb{R}}
\frac{(1+d_1)p_n(x-y)}{1+d_1(M p_n(x-y)+(b_1+a_1)\epsilon_1 )}k_1(y)dy,
\]
and $p_0(x)>0,x\in \mathbb{R}.$

Let $c_3=\frac{2c_2+c^*}{3} < \frac{c_2+c^*}{2},$ then $c_3>c_2.$ Using Lemma \ref{le3}, we obtain
\[
\liminf_{n\to + \infty}\inf_{|x|=c_3 n} p_n(x) \ge \frac{1-(a_1+b_1)\epsilon_1}{M}>0.
\]
However, if $-x=c_3 n$, then $\xi=(c_2-c_3)n\to -\infty$ as $ n\to + \infty ,$ and \eqref{5.a} implies that
\[
\limsup_{n\to + \infty}\sup_{-x=c_3 n} p_n(x) =0,
\]
which implies a contradiction. The proof is complete.
\end{proof}

\subsection{$m$ Species Competition System}
\noindent

We now consider the traveling wave solutions of \eqref{1} if
\begin{equation}
P_{i}=\frac{\left( 1+d_{i}\right) u_{n}^{i}}{1+d_{i}\left(
u_{n}^{i}+\sum_{j\in J^{\prime }}e_{j}^{i}u_{n-j}^{i}+\sum_{j\in J,l\in
I,l\neq i}f_{lj}^{i}u_{n-j+1}^{l}\right) },  \label{n1}
\end{equation}
in which all the parameters are nonnegative and $d_{i}>0,i\in I,J^{\prime
}=\{1,2,\cdots ,\tau-1\}.$ If
\begin{equation}
\sum_{j\in J^{\prime }}e_{j}^{i}+\sum_{j\in J,l\in I,l\neq i}f_{lj}^{i}<1,
\label{n2}
\end{equation}%
then \eqref{1} with \eqref{n1} has a spatially homogeneous steady state
\[
\mathbf{E}=(E_{1},\cdots ,E_{m})\gg \mathbf{0}
\]%
with%
\[
E_{i}+\sum_{j\in J^{\prime }}e_{j}^{i}E_{i}+\sum_{j\in J,l\in I,l\neq
i}f_{lj}^{i}E_{l}=1,i\in I.
\]%
Let $\Phi =(\phi _{1},\phi _{2},\cdots ,\phi _{m})$ with
$
u_{n}^{i}(x)=\phi _{i}(\xi ),\xi =x+cn,i\in I,
$
be a traveling wave solution of \eqref{1} with \eqref{n1}, then $\Phi (\xi )$
and $c$ satisfy
\begin{equation}
\phi _{i}(\xi )=\int_{\mathbb{R}}\frac{\left( 1+d_{i}\right) \phi
_{i}(y)k_{i}(\xi -c-y)}{1+d_{i}\left( \phi _{i}(y)+\sum_{j\in J^{\prime
}}e_{j}^{i}\phi _{i}(y-cj)+\sum_{j\in J,l\in I,l\neq i}f_{lj}^{i}\phi
_{l}(y-cj+c)\right) }dy  \label{tt}
\end{equation}%
for all $i\in I,\xi \in \mathbb{R}.$

When $\lambda \ge 0, c\ge 0, $ define
\[
\Lambda_i(\lambda, c)=(1+d_{i})\int_{\mathbb{R}}e^{\lambda y- \lambda c} k_i(y)dy, i\in I.
\]
\begin{lemma}
Define
\[
C= \max_{i\in I}\left\{\inf_{\lambda >0}\frac{\ln \left((1+d_{i})\int_{\mathbb{R}}e^{\lambda y}k_i(y)dy\right)}{\lambda}\right\}.
\]
\begin{description}
\item[(R1)] If $c>C,$ then for each $i\in I,$ there exists $\lambda_i^c>0$ satisfying
$
\Lambda_i(\lambda_i^c, c)=1;
$
\item[(R2)] for each fixed $c>C,$ there exists $\eta \in (1,2)$ such that
\[
\Lambda_i(\lambda, c)>1, \lambda \in (0, \lambda_i^c) \text{ and } \Lambda_i(\lambda, c)<1, \lambda \in (\lambda_i^c, \eta\lambda_i^c), i\in I;
\]
\item[(R3)] when $c\in (0,C),$ there exists $i_0\in I$ such that $\Lambda_{i_0}(\lambda, c)>1, \lambda >0.$
\end{description}
\end{lemma}
\begin{theorem}\label{thm1}
If $c>C,$ then \eqref{tt}  has a bounded positive solution satisfying
\[
0< \phi_i(\xi) <1, \xi\in \mathbb{R}, \lim_{\xi\to - \infty}\phi_i(\xi)e^{-\lambda_i^c \xi}=1, i\in I.
\]
\end{theorem}
\begin{proof}
Define continuous functions as follows
\[
\overline{\phi }_{i}(\xi )=\min \{e^{\lambda _{i}^c\xi },1\},\underline{%
\phi }_{i}(\xi )=\max \{e^{\lambda _{i}^c\xi }-Ne^{\eta \lambda _{i}^c\xi
},0\},i\in I,\xi\in\mathbb{R}.
\]%
Then $(\overline{\phi }_{1}(\xi
),\cdots ,\overline{\phi }_{m}(\xi )),(\underline{\phi }_{1}(\xi ),\cdots ,%
\underline{\phi }_{1}(\xi ))$ are a pair of generalized upper and lower solutions of %
\eqref{tt} if $N>1$ is large enough.

We now verified the definition of generalized upper and lower solutions for
any uniformly continuous functions $\psi _{ij}(\xi )$ satisfying
\[
\underline{\phi }_{i}(\xi )\leq \psi _{ij}(\xi )\leq \overline{\phi }%
_{i}(\xi ),i\in I,j\in J,\xi \in \mathbb{R}.
\]%
In fact, from the monotonicity of $P_{i}$ in different variables and the
monotonicity of $\overline{\phi }_{i}(\xi )$ in $\xi\in\mathbb{R}$, it suffices to verify that%
\begin{equation}
\underline{\phi }_{i}(\xi )\leq \int_{\mathbb{R}}\frac{\left( 1+d_{i}\right)
\underline{\phi }_{i}(y)k_{i}(\xi -y-c)}{1+d_{i}\left( \underline{\phi }%
_{i}(y)+\sum_{j\in J^{\prime }}e_{j}^{i}\overline{\phi }_{i}(y)+\sum_{j\in
J,l\in I,l\neq i}f_{lj}^{i}\overline{\phi }_{l}(y)\right) }dy,\xi \in
\mathbb{R},i\in I,  \label{sub1}
\end{equation}%
and%
\begin{equation}
\int_{\mathbb{R}}\frac{\left( 1+d_{i}\right) \overline{\phi }_{i}(y)}{1+d_{i}%
\overline{\phi }_{i}(y)}k_{i}(\xi -y-c)dy\leq \overline{\phi }_{i}(\xi ),\xi
\in \mathbb{R},i\in I  \label{sup1}
\end{equation}%
because of%
\begin{eqnarray*}
&&\frac{\left( 1+d_{i}\right) \underline{\phi }_{i}(y)}{1+d_{i}\left(
\underline{\phi }_{i}(y)+\sum_{j\in J^{\prime }}e_{j}^{i}\overline{\phi }%
_{i}(y)+\sum_{j\in J,l\in I,l\neq i}f_{lj}^{i}\overline{\phi }_{l}(y)\right)
} \\
&\leq &\frac{\left( 1+d_{i}\right) \underline{\phi }_{i}(y)}{1+d_{i}\left(
\underline{\phi }_{i}(y)+\sum_{j\in J^{\prime }}e_{j}^{i}\overline{\phi }%
_{i}(y-cj)+\sum_{j\in J,l\in I,l\neq i}f_{lj}^{i}\overline{\phi }%
_{l}(y-cj+c)\right) } \\
&\leq &\frac{\left( 1+d_{i}\right) \psi _{i1}(y)}{1+d_{i}\left( \psi
_{i1}(y)+\sum_{j-1\in J^{\prime }}e_{j-1}^{i}\psi _{ij}(y-cj+c)+\sum_{j\in
J,l\in I,l\neq i}f_{lj}^{i}\psi _{lj}(y-cj+c)\right) } \\
&\leq &\frac{\left( 1+d_{i}\right) \overline{\phi }_{i}(y)}{1+d_{i}\left(
\overline{\phi }_{i}(y)+\sum_{j-1\in J^{\prime }}e_{j-1}^{i}\psi
_{ij}(y-cj+c)+\sum_{j\in J,l\in I,l\neq i}f_{lj}^{i}\psi
_{lj}(y-cj+c)\right) } \\
&\leq &\frac{\left( 1+d_{i}\right) \overline{\phi }_{i}(y)}{1+d_{i}\overline{%
\phi }_{i}(y)},y\in \mathbb{R}.
\end{eqnarray*}

If $\overline{\phi }_{i}(\xi )=1,$ then%
\[
\int_{\mathbb{R}}\frac{\left( 1+d_{i}\right) \overline{\phi }_{i}(y)}{1+d_{i}%
\overline{\phi }_{i}(y)}k_{i}(\xi -y-c)dy\leq 1=\overline{\phi }_{i}(\xi )
\]%
is evident for each $i\in I.$

If $\overline{\phi }_{i}(\xi )=e^{\lambda _{i}^{c}\xi },$ then%
\begin{eqnarray*}
\int_{\mathbb{R}}\frac{\left( 1+d_{i}\right) \overline{\phi }_{i}(y)}{%
1+d_{i}\overline{\phi }_{i}(y)}k_{i}(\xi -y-c)dy &\leq & \int_{\mathbb{R}}\left(
1+d_{i}\right) \overline{\phi }_{i}(y)k_{i}(\xi -y-c)dy \\
&\leq &\int_{\mathbb{R}}\left( 1+d_{i}\right) e^{\lambda _{i}^{c}y}k_{i}(\xi
-c-y)dy \\
&=&e^{\lambda _{i}^{c}\xi }=\overline{\phi }_{i}(\xi )
\end{eqnarray*}%
for each $i\in I.$ Therefore, \eqref{sup1} holds.

If $\underline{\phi }_{i}(\xi )=0,$ then
\[
\frac{\left( 1+d_{i}\right) \underline{\phi }_{i}(y)}{1+d_{i}\left(
\underline{\phi }_{i}(y)+\sum_{j\in J^{\prime }}e_{j}^{i}\overline{\phi }%
_{i}(y)+\sum_{j\in J,l\in I,l\neq i}f_{lj}^{i}\overline{\phi }_{l}(y)\right)
}\geq 0,y\in \mathbb{R}
\]%
and%
\[
\int_{\mathbb{R}}\frac{\left( 1+d_{i}\right) \underline{\phi }%
_{i}(y)k_{i}(\xi -y-c)}{1+d_{i}\left( \underline{\phi }_{i}(y)+\sum_{j\in
J^{\prime }}e_{j}^{i}\overline{\phi }_{i}(y)+\sum_{j\in J,l\in I,l\neq
i}f_{lj}^{i}\overline{\phi }_{l}(y)\right) }dy\geq 0=\underline{\phi }%
_{i}(\xi )
\]%
for each $i\in I.$

If $\underline{\phi }_{i}(\xi )=e^{\lambda _{i}^{c}\xi }-Ne^{\eta \lambda
_{i}^{c}\xi },$ then
$
\frac{1}{1+a}\geq 1-a,a\geq 0
$
indicates that
\begin{eqnarray*}
&&\frac{\left( 1+d_{i}\right) \underline{\phi }_{i}(y)}{1+d_{i}\left(
\underline{\phi }_{i}(y)+\sum_{j\in J^{\prime }}e_{j}^{i}\overline{\phi }%
_{i}(y)+\sum_{j\in J,l\in I,l\neq i}f_{lj}^{i}\overline{\phi }_{l}(y)\right)
} \\
&\geq &\left( 1+d_{i}\right) \underline{\phi }_{i}(y) \\
&&-d_{i}\left( 1+d_{i}\right) \underline{\phi }_{i}(y)\left( \underline{\phi
}_{i}(y)+\sum_{j\in J^{\prime }}e_{j}^{i}\overline{\phi }_{i}(y)+\sum_{j\in
J,l\in I,l\neq i}f_{lj}^{i}\overline{\phi }_{l}(y)\right)  \\
&\geq &\left( 1+d_{i}\right) \underline{\phi }_{i}(y) \\
&&-d_{i}\left( 1+d_{i}\right) e^{\lambda _{i}^{c}y}\left[ \left(
1+\sum_{j\in J^{\prime }}e_{j}^{i}\right) e^{\lambda _{i}^{c}y}+\sum_{j\in
J,l\in I,l\neq i}f_{lj}^{i}e^{\lambda _{l}^{c}y}\right]  \\
&\geq &\left( 1+d_{i}\right) \left( e^{\lambda _{i}^{c}y}-Ne^{\eta \lambda
_{i}^{c}y}\right)  \\
&&-d_{i}\left( 1+d_{i}\right) \left[ \left( 1+\sum_{j\in J^{\prime
}}e_{j}^{i}\right) e^{2\lambda _{i}^{c}y}+\sum_{j\in J,l\in I,l\neq
i}f_{lj}^{i}e^{(\lambda _{l}^{c}+\lambda _{i}^{c})y}\right]
\end{eqnarray*}%
and%
\begin{eqnarray*}
&&\int_{\mathbb{R}}\frac{\left( 1+d_{i}\right) \underline{\phi }%
_{i}(y)k_{i}(\xi -y-c)}{1+d_{i}\left( \underline{\phi }_{i}(y)+\sum_{j\in
J^{\prime }}e_{j}^{i}\overline{\phi }_{i}(y)+\sum_{j\in J,l\in I,l\neq
i}f_{lj}^{i}\overline{\phi }_{l}(y)\right) }dy \\
&\geq &\int_{\mathbb{R}}\left( 1+d_{i}\right) \left( e^{\lambda
_{i}^{c}y}-Ne^{\eta \lambda _{i}^{c}y}\right) k_{i}(\xi -c-y)dy \\
&&-d_{i}\int_{\mathbb{R}}\left( 1+d_{i}\right) \left[ \left( 1+\sum_{j\in
J^{\prime }}e_{j}^{i}\right) e^{2\lambda _{i}^{c}y}+\sum_{j\in J,l\in
I,l\neq i}f_{lj}^{i}e^{(\lambda _{l}^{c}+\lambda _{i}^{c})y}\right]
k_{i}(\xi -c-y)dy \\
&=&e^{\lambda _{i}^{c}\xi }-N\Lambda _{i}(\eta \lambda _{i}^{c},c)e^{\eta
\lambda _{i}^{c}\xi } \\
&&-d_{i}\left( 1+\sum_{j\in J^{\prime }}e_{j}^{i}\right) \Lambda
_{i}(2\lambda _{i}^{c},c)e^{2\lambda _{i}^{c}\xi }-d_{i}\sum_{j\in J,l\in
I,l\neq i}f_{lj}^{i}\Lambda _{i}(\lambda _{l}^{c}+\lambda
_{i}^{c},c)e^{(\lambda _{l}^{c}+\lambda _{i}^{c})y} \\
&\geq &\underline{\phi }_{i}(\xi )=e^{\lambda _{i}^{c}\xi }-Ne^{\eta \lambda
_{i}^{c}\xi }
\end{eqnarray*}%
for each $i\in I$ provided that
\[
N=\max_{i\in I}\left\{ \frac{d_{i}\left( 1+\sum_{j\in J^{\prime
}}e_{j}^{i}\right) \Lambda _{i}(2\lambda _{i}^{c},c)+d_{i}\sum_{j\in J,l\in
I,l\neq i}f_{lj}^{i}\Lambda _{i}(\lambda _{l}^{c}+\lambda _{i}^{c},c)}{%
1-\Lambda _{i}(\eta \lambda _{i}^{c},c)}\right\} +1,
\]
which implies \eqref{sub1} for each $i\in I$.

Therefore, we have obtained a pair of generalized upper and lower solutions.
By Theorem \ref{th1}, \eqref{tt} has a bounded positive solution satisfying
\[
0\leq \phi _{i}(\xi )\leq 1,\xi \in \mathbb{R},\lim_{\xi \rightarrow -\infty
}\phi _{i}(\xi )e^{-\lambda _{i}^{c}\xi }=1,i\in I.
\]%
Due to (k1) and similar to Theorem \ref{th5}, we further have
\[
0<\phi _{i}(\xi )<1,\xi \in \mathbb{R},i\in I.
\]%
The proof is complete.
\end{proof}
\begin{theorem}
Assume that \eqref{n2} and Theorem \ref{thm1} hold. Then
\[
\lim_{\xi\to-\infty}\phi_i (\xi)=0, \lim_{\xi\to + \infty}\phi_i (\xi)=E_i, i\in I.
\]
\end{theorem}
\begin{proof}
Note that
\[
u_n^i(x)=\phi_i(x+cn), x\in\mathbb{R}, n=0,1,2,\cdots, i\in I
\]
satisfies
\[
u_{n+1}^i(x)\ge \int_{\mathbb{R}}\frac{(1+d_i)u_n^i(y)}{1+d_i(u_n^i(y)+ \sum_{j\in J'}e_{j}^{i}+\sum_{j\in J,l\in I,l\neq i}f_{lj}^{i})}k_i(y)dy.
\]
Then, similar to the proof of Theorem \ref{th6}, we can prove that
\[
\liminf_{\xi\to + \infty}\phi_i (\xi) \ge 1- \left( \sum_{j\in J'}e_{j}^{i}+\sum_{j\in J,l\in I,l\neq i}f_{lj}^{i} \right) >0, i\in I.
\]
Define
\[
r_i(s)=sE_i, t_i(s)=sE_i +(1-s)(1+\epsilon),i\in I
\]
with $\epsilon >0$ small enough. Similar to the proof of Lemma \ref{le5}, we can verify
that
\[
\begin{cases}
r_{i}(s)<\frac{\left( 1+d_{i}\right) r_{i}(s)}{1+d_{i}\left(
r_{i}(s)+t_{i}(s)\sum_{j\in J^{\prime }}e_{j}^{i}+\sum_{j\in J,l\in I,l\neq
i}f_{lj}^{i}t_{l}(s)\right) },s\in (0,1),\\
t_{i}(s)>\frac{\left(
1+d_{i}\right) t_{i}(s)}{1+d_{i}\left( t_{i}(s)+r_{i}(s)\sum_{j\in J^{\prime
}}e_{j}^{i}+\sum_{j\in J,l\in I,l\neq i}f_{lj}^{i}r_{l}(s)\right) },s\in
(0,1),
\end{cases}
\]
and so $r_i(s), t_i (s)$ satisfy
the definition of contracting rectangle. By Theorem \ref{th3}, we complete the proof.
\end{proof}

We now present the nonexistence of traveling wave solutions without proof because the proof is similar to that of Theorem \ref{th5.}.
\begin{theorem}\label{thm2}
If $c<C,$ then \eqref{tt} has not a bounded positive  solution
such that
\[
\liminf_{\xi\to + \infty}\phi_i(\xi)>0, \lim_{\xi\to - \infty}\phi_i(\xi)=0, i\in I.
\]
\end{theorem}
Clearly, Theorems \ref{thm1}-\ref{thm2} remain true for the following equation
\begin{equation}\label{li1}
w_{n+1}(x)=\int_{\mathbb{R}}
\frac{(1+d)w_n(x-y)}{1+d(w_n(x-y)+a w_{n-1}(x-y))}k(y)dy,
\end{equation}
in which $x\in\mathbb{R}, n=0,1,2, \cdots,$  $k(y)$ satisfies (k1)-(k3), $ d>0$ and $ a\in [0,1)$ are constants. More precisely, let
\[
c^*= \inf_{\lambda >0}\frac{\ln \left((1+d)\int_{\mathbb{R}}e^{\lambda y}k(y)dy\right)}{\lambda} .
\]
If $c> c^*,$ then \eqref{li1} has a traveling wave solution
$
w_n(x)=\phi (\xi), x+cn=\xi
$
such that
\begin{equation}\label{li2}
0 < \phi (\xi) \le 1, \xi\in\mathbb{R}, \lim_{\xi\to-\infty} \phi (\xi) =0, \lim_{\xi\to + \infty} \phi (\xi) =\frac{1}{1+a}.
\end{equation}
When $c\in (0,c^*),$ then \eqref{li1} does not admit a traveling wave solution satisfying \eqref{li2}. For $c=c^*,$ our main conclusion is given as follows.
\begin{theorem}
If $c=c^*,$ then \eqref{li1} has a traveling wave solution satisfying \eqref{li2}.
\end{theorem}
\begin{proof}
Similar to the proof of Theorem \ref{th4}, we can prove the existence of traveling wave solutions of \eqref{li1} by passing to a limit function. More precisely,
let $\{c_{l}\}_{l\in \mathbb{N}}$ be a strictly decreasing
sequence satisfying $\lim_{l\rightarrow \infty }c_{l}=c^*$  and for each $l\in \mathbb{N},$ $\phi^{l}(\xi )$ is a traveling wave solution of \eqref{li1} with wave speed $c_l$ such that
\begin{equation*}
\phi ^{l}(\xi )<\frac{1-a}{8},\xi <0,\phi ^{l}(0 )=\frac{1-a}{8},
\end{equation*}
then
there exists $\phi (x+c^*n)$ satisfying \eqref{li1} and
\begin{equation}\label{asfu}
\phi(0)=\frac{1-a}{8}, \phi (\xi) \le \frac{1-a}{8}, \xi \le 0
\end{equation}
and
$
0\le \phi (\xi)\le 1, \xi \in \mathbb{R}.
$
From (k1), we see that $\phi (\xi)$ is uniformly continuous in $\xi \in \mathbb{R}.$

Since
$w_n(x)=\phi (\xi)$ satisfies
\begin{equation*}
\begin{cases}
w_{n+1}(x)\ge \int_{\mathbb{R}}
\frac{(1+d)w_n(x-y)}{1+d(w_n(x-y)+a)}k(y)dy,x\in\mathbb{R}, n=0,1,2, \cdots, \\
w_0(x)=\phi (x), w_0(0)=\frac{1-a}{8},x\in\mathbb{R},
\end{cases}
\end{equation*}
then similar to the proof of Theorem \ref{th6}, we have
\[
\liminf_{\xi\to + \infty} \phi (\xi) \ge 1-a >0.
\]
Denote
\[
\liminf_{\xi\to + \infty} \phi (\xi)=\phi^-, \limsup_{\xi\to + \infty} \phi (\xi)=\phi^+.
\]
Then
\[
\phi^-, \phi^+ \in (0,1].
\]
From the dominated convergence theorem and monotonicity, we have
\[
\phi ^{+}\leq \frac{(1+d)\phi ^{+}}{1+d(\phi ^{+}+a\phi ^{-})},\,\,\,\, \phi ^{-}\geq
\frac{(1+d)\phi ^{-}}{1+d(\phi ^{-}+a\phi ^{+})}
\]
and so
\[
\phi ^{+}= \phi ^{-}=\lim_{\xi\to + \infty} \phi (\xi)=\frac{1}{1+a}.
\]

We now prove that $\lim_{\xi\to-\infty} \phi (\xi) =0.$
Let
$
\limsup_{\xi\to -\infty} \phi (\xi)=\vartheta,
$
then $\vartheta \in [0, \frac{1-a}{8}]$ by \eqref{asfu}.
If $\vartheta >0,$ then the uniform continuity implies that there exist a strictly decreasing sequence $\{\xi_l\}_{l\in \mathbb{N}}$ satisfying
\[
\xi_l <0 \text{ for }l\in\mathbb{N} \text{ and }\lim_{l\to + \infty} \xi_l = -\infty,
\]
and a constant $\delta >0$ such that
\[
\phi (\xi)> \vartheta /2, \xi \in [\xi_l- \delta, \xi_l + \delta],l\in \mathbb{N}.
\]

Consider the initial value problem
\begin{equation*}
\begin{cases}
\underline{w}_{n+1}(x)= \int_{\mathbb{R}}
\frac{(1+d)\underline{w}_n(x-y)}{1+d(\underline{w}_n(x-y)+a)}k(y)dy,x\in\mathbb{R}, n=0,1,2, \cdots, \\
\underline{w}_0(x)= \chi (x),x\in\mathbb{R},
\end{cases}
\end{equation*}
in which $\chi (x)$ is a continuous function and satisfies
\begin{description}
  \item[(x1)] $\chi (x)=0, |x|\ge  \delta;$
  \item[(x2)] $\chi (x)$ is decreasing for $x\in [\delta /2, \delta];$
  \item[(x3)] $\chi (x)=\chi (-x),x\in\mathbb{R}; $
  \item[(x4)] $\chi (x)=\vartheta /2, |x|<  \delta /2.$
\end{description}
By Lemma \ref{le3}, there exists $T\in\mathbb{N}$ such that
\[
\underline{w}_{n}(0)>\frac{1-a}{2} , n\ge  T
\]
and
\[
w_n(\xi_l)\ge \underline{w}_n(0), x\in\mathbb{R}, n\in\mathbb{N}, l\in\mathbb{N}.
\]
We then obtain that
\[
w_T(\xi_l)= \phi (\xi_l +c^*T) >\frac{1-a}{2}, l\in \mathbb{N},
\]
and a contradiction occurs because of
\[
\limsup_{\xi\to -\infty} \phi (\xi)=\vartheta < \frac{1-a}{4}.
\]
The proof is complete.
\end{proof}

\section*{Acknowledgments}
I am grateful to the anonymous referee for his/her valuable suggestions which led to an improvement
of the original manuscript. Research  was supported by NSF of China (Grant No. 11101194).

\end{document}